\documentclass[12pt]{article}
\usepackage{setspace} 
\usepackage[dvips]{graphicx} 
\usepackage{amsmath,amsthm,amsfonts}
\usepackage{epsfig}
\usepackage{subfig}
\RequirePackage[colorlinks,citecolor=blue,urlcolor=blue,linkcolor=blue]{hyperref}
\hypersetup{
colorlinks = true,
citecolor=blue,
urlcolor=blue,
linkcolor=blue,
pdfauthor = {Alexey Kuznetsov},
pdfkeywords = {stable process, supremum, q-series, continued fractions, Diophantine approximations, double gamma function, Barnes function, dilogarithm, q-Pochhammer symbol, quantum dilogarithm},
pdftitle = {On the density of the supremum of a stable process},
pdfpagemode = UseNone
}
    \def\qed{\hfill$\sqcap\kern-8.0pt\hbox{$\sqcup$}$\\}
    \def\beq{\begin{eqnarray}}
    \def\eeq{\end{eqnarray}}
    \def\beqq{\begin{eqnarray*}}
    \def\eeqq{\end{eqnarray*}}
    \def\zz{{\mathbb Z}}
    \def\n{{\mathbb N}}
    
    \def\q{{\mathbb Q}}
    
    \def\re{\textnormal {Re}}
    \def\im{\textnormal {Im}}
    \def\p{{\mathbb P}}
    \def\e{{\mathbb E}}
    \def\r{{\mathbb R}}
    \def\aa{\mathcal A}
    \def\ll{{\mathcal L}}
    \def\tl{\tilde {\mathcal L}}
    \def\mm{{\mathcal M}}
    \def\c{{\mathbb C}}

    \def\Li{{\textnormal{Li}}}	
    \def\d{{\textnormal d}}
    \def\i{{\textnormal i}}

        \def\ee{{\textnormal e}} 
	\newtheorem{theorem}{Theorem}
	
	\newtheorem{proposition}{Proposition}
	
	\newtheorem{definition}{Definition}

\title{On the density of the supremum of a stable process}
\author{
A. Kuznetsov
\thanks{{Research supported by the
Natural Sciences and Engineering Research Council of Canada.}}  \\ \\
Dept. of Mathematics and Statistics\\  York University \\
4700 Keele Street 
\\Toronto, ON \\ M3J 1P3,  Canada 
 }

\date{Current version: \today}

\begin{document}

%****************************************************************************************************************
%****************************************************************************************************************
%****************************************************************************************************************

\maketitle

\begin{abstract}
\bigskip
 We study the density of the supremum of a strictly stable L\'evy process.
As was proved recently in \cite{HubKuz2011}, for almost all irrational values of the stability parameter $\alpha$ this 
density can be represented by an absolutely convergent series.
We show that this result is not valid  for {\it all} irrational values of $\alpha$: we construct a dense uncountable set of irrational numbers $\alpha$ for which the series does not converge absolutely. 
Our second goal is to investigate in more detail the important case when $\alpha$ is rational. 
We derive an explicit formula for the Mellin transform of the supremum, which is given in terms of Gamma function and dilogarithm. In order to illustrate the usefulness of these results we perform several numerical experiments and discuss their implications. Finally, we state some interesting connections that this problem has to other areas of Mathematics and Mathematical Physics, such as q-series, Diophantine approximations
and quantum dilogarithms, and we also suggest several open problems.   
\end{abstract}

{\vskip 0.5cm}
 \noindent {\it Keywords}: stable process, supremum, q-series, q-Pochhammer symbol, continued fractions, Diophantine approximations,  double gamma function, Barnes function, dilogarithm, quantum dilogarithm
{\vskip 0.5cm}
 \noindent {\it 2000 Mathematics Subject Classification }: 60G52 

\newpage

%****************************************************************************************************************
%****************************************************************************************************************
%****************************************************************************************************************

%****************************************************************************************************************
%****************************************************************************************************************
%****************************************************************************************************************

\section{Introduction}

%****************************************************************************************************************
%****************************************************************************************************************
%****************************************************************************************************************

Let $X$ be a strictly stable L\'evy process, which is started at zero and is described by the stability parameter $\alpha \in (0,1)\cup(1,2)$
and the skewness parameter $\beta \in [-1,1]$. We assume that $|X|$ is not a subordinator, which implies that 
$|\beta|\ne 1$ if $\alpha \in (0,1)$. It is well-known that the  
characteristic exponent $\Psi(z):=-\ln(\e[\exp(\i z X_1)])$ is given by
\beq\label{def_Psi0}
\Psi(z)=c|z|^{\alpha} \left(1-\i \beta \tan\left(\frac{\pi \alpha}2\right) {\textnormal {sign}}(z)\right), \;\;\; z\in \r,
\eeq
where $c$ is a positive constant (see Theorem 14.15 in \cite{Sato} or Chapter 8 in \cite{Bertoin}). The characteristic exponent $\Psi(z)$ satisfies 
$\Psi(a^{\frac{1}{\alpha}}z)=a \Psi(z)$ for all $a>0$, which implies the important scaling property: the 
processes $\{X_{at}\; : t\ge 0\}$ and $\{a^{\frac{1}{\alpha}}X_t\; : t\ge 0\}$ have the same distribution. 
In particular, without loss of generality we can assume that the parameter $c$ in \eqref{def_Psi0} is normalized so that
\beq\label{condition_c}
c^{2}\left(1+\beta^2  \tan\left(\frac{\pi \alpha}2\right)^2\right)=1.
\eeq

As we will see later, it is more natural to parametrize stable processes 
by parameters $(\alpha,\rho)$ instead of $(\alpha,\beta)$, 
where the positivity parameter $\rho:=\p(X_1>0)$ can be expressed in terms of $(\alpha,\beta)$ as follows 
\beq\label{def_rho}
\rho=\frac{1}{2}+\frac{1}{\pi\alpha} \tan^{-1} \left(\beta \tan\left(\frac{\pi \alpha}2 \right)\right),
\eeq
see Section 2.6 in \cite{Zolotarev1986}.
Using Theorem 14.15 in \cite{Sato} one can check that with our normalization \eqref{condition_c} the L\'evy measure of $X$ is  
given by
\beq\label{def_Levy_measure}
\Pi(\d x) =  \frac{|x|^{-1-\alpha}}{2\sin \left(\frac{\pi \alpha}2 \right)}
\bigg[ \sin(\pi \alpha \rho) {\mathbf 1}_{\{x>0\}} + 
  \sin(\pi \alpha (1-\rho)) {\mathbf 1}_{\{x<0\}} \bigg] \d x.
\eeq
It is also easy to see that under transformation \eqref{def_rho} the set of parameters $(\alpha,\beta)$ that we have described above is bijectively mapped onto a set 
\beqq
\aa:=\{\alpha \in (0,1), \; \rho \in (0,1)\} \cup  \{\alpha\in(1,2), \; \rho \in [1-\alpha^{-1}, \alpha^{-1}]\},
\eeqq
which is shown on Figure \ref{fig1}. 
We call the set ${\mathcal A}$ an {\it admissible set of parameters}. Note that formula
\eqref{def_Levy_measure} implies that when $\alpha \in (1,2)$ and $\rho =1-\alpha^{-1}$ \{$\rho=\alpha^{-1}$\} the process $X$ 
is spectrally-positive \{resp. spectrally-negative\}.

We define the supremum of the process $X$ as $S_t:=\sup\{X_u\; : 0\le u \le t\}$. 
Note that due to the scaling property of stable processes, we have $S_t \stackrel{d}{=} t^{\frac{1}{\alpha}} S_1$, 
thus it is sufficient to study the distribution of $S_1$.
Our main object of interest is the density  of $S_1$, which will be denoted by $p(x;\alpha,\rho)$ (or by $p(x)$ if we do not need to stress the dependence on parameters).

The study of the distribution of $S_1$ and its various integral transforms is an interesting problem that has a long history. 
As the following identity shows, an equivalent problem is to study the distribution of the first passage time $\tau_h^+:=\inf \{ t > 0 \; : \; X_t>h \}$
for any positive $h$ 
\beqq
\p(\tau_h^+<t)=\p(S_t>h)=\p\left(S_1 > h t^{-\frac{1}{\alpha}} \right).
\eeqq
Another closely related problem is the investigation of the Wiener-Hopf factorization for stable processes. The positive
Wiener-Hopf factor is defined as 
\beqq
\phi(z)=\phi(z;\alpha,\rho):=\e\left[e^{-zS_{\ee(1)}}\right], \;\;\; \re(z)\ge 0, 
\eeqq 
where $\ee(q)$ is independent of $X$ and has exponential distribution with  $\e[\ee(q)]=1/q$. We see that the Wiener-Hopf factor gives us information about the supremum evaluated at an exponential time  $\ee(q)$, while we would like to obtain the infomation about the supremum at the deterministic time. The following identity,  which is based on the scaling property, connects these two problems 
\beq\label{piece1}
\int\limits_0^{\infty} z^{w-1}\e\left[e^{-zS_{\ee(1)}} \right] \d z=
\Gamma(w) \e\left[ \left(S_{\ee(1)}\right)^{-w}\right]=\Gamma(w) \e\left[ \ee(1)^{-\frac{w}{\alpha}} \left(S_1\right)^{-w} \right]=
\Gamma(w) \Gamma\left( 1-\frac{w}{\alpha}\right)\e\left[ \left(S_1\right)^{-w} \right].
\eeq
The second piece of the puzzle is provided by the integral representation for the Wiener-Hopf factor
\beq\label{piece2}
\ln\left(\e\left[e^{-zS_{\ee(q)}} \right]\right)=\frac{z}{2\pi } \int\limits_{\r}   \ln\left(\frac{q}{q+\Psi(u)}\right) \frac{\d u}{u(u-\i z)}, \;\;\; \re(z)>0.
\eeq
Formula \eqref{piece2} is in fact valid for all L\'evy processes satisfying a mild regularity condition (see Lemma 4.2 in \cite{Mordecki}).

It seems that Darling was the first to investigate the supremum of a stable process. 
In the paper \cite{Darling1956} which was published in 1956 he has obtained \eqref{piece1} and \eqref{piece2} for symmetric stable processes ($\rho=1/2$). As an application of his results Darling has found a 
simple expression for the density of the supremum of Cauchy process ($(\alpha,\rho)=(1,1/2)$).
Then in 1969 Heyde \cite{Heyde1969} has shown that formulas \eqref{piece1} and \eqref{piece2} are also valid for general non-symmetric 
stable processes. Many interesting results were discovered by Bingham and were published in his 1973 paper \cite{Bingham1973}.
Bingham has also used an approach based on \eqref{piece1} and \eqref{piece2} and has obtained an absolutely convergent series representation  for the density of $S_1$ when the process $X$ is spectrally negative (in this case the distribution of $S_1$ is related to Mittag-Leffler functions). 
Note that the spectrally-negative case is rather simple, as it follows from
the general theory that the supremum $S_{\ee(q)}$ for any spectrally-negative L\'evy process must necessarily have exponential distribution (see formula (8.2) in \cite{Kyprianou}).  In the general case Bingham has obtained the first asymptotic term
for $\p(S_1\le x)$ as $x\to 0^+$. 
His results also  provide the first indication of the fact
that the analytic structure of the distribution of the supremum may be related to certain arithmetic properties of the parameters: 
he derives a general asymptotic result for the Wiener-Hopf factor $\phi(z)$, and this result is valid only if 
$\alpha$ is irrational (see formula (9) in \cite{Bingham1973}).

Bingham \cite{Bingham1973} states that in general Darling's integral (the integral in the right-hand side of \eqref{piece2})
can not be evaluated explicitly, and that the approach based on \eqref{piece1} and \eqref{piece2} can only succeed
in the spectrally-negative case.  
That this prediction was not entirely correct was shown in 1987 by Doney \cite{Doney1987}. Among other results, Doney has shown that Darling's integral can be evaluated explicitly if
the process $X$ belongs to the class 
${\mathcal C}_{k,l}$, which is defined by the condition 
\beq\label{Ckl}
\rho+k=\frac{l}{\alpha},
\eeq
where $k$ and $l$ are integers.  These classes of processes can be considered as generalizations of the spectrally one-sided processes, for which $\rho=1/\alpha$ or $\rho=1-1/\alpha$. It is easy to see that the curves 
$\rho=\{l/\alpha\}$ (here $\{x\}\in [0,1)$ denotes the fractional part of $x\in \r$) are dense in the set of admissible parameters ${\mathcal A}$. These curves in fact create a rather interesting pattern (see Figure \ref{fig1}), and as we will see later, the dark/white regions on this graph correspond to different series representations for the density 
$p(x;\alpha,\rho)$.  
The results by Doney clearly illustrate the importance of the arithmetic properties of parameters. 
We see that depending on whether $\alpha$ is rational or not the formula \eqref{Ckl} defines a countable/dense or a finite set of $\rho$ for  
which there exists an explicit formula for the Wiener-Hopf factor $\phi(z)$.

In the 1990s and 2000s  there have appeared many other important results related to extrema of stable processes. 
Bertoin \cite{Bertoin01091996} has computed the Laplace transform of the first exit time from a finite interval in the spectrally 
one-sided case, the result is given in terms of Mittag-Leffler functions. 
 An interesting duality result relating space-time Wiener-Hopf factorization of stable processes with parameters 
$(\alpha,\rho)$ and $(1/\alpha,\alpha \rho)$ was discovered by Fourati \cite{Fourati}, a special case of this duality was given earlier by Doney (see Theorem 3 in \cite{Doney1987}). 
An absolutely convergent series representation for $p(x)$ was obtained by 
 Bernyk, Dalang and Peskir \cite{Bernyk2008} in the spectrally-positive case. Doney \cite{Doney2008}
 has found the first asymptotic term of $p(x)$ as $x\to +\infty$ in the spectrally-positive case, 
 and Patie \cite{Patie2009} has given a complete asymptotic expansion.
The spectrally positive case and various connections with Mittag-Leffler functions were investigated by Simon \cite{Simon2010}. 
In the general case the first term of the asymptotic expansion of $p(x)$ as $x \to 0^+$ or $x\to+\infty$ was obtained by Doney and Savov \cite{Doney2010}. Peskir \cite{Peskir2008} and Simon \cite{Simon2011} study 
the first hitting time of spectrally-positive stable processes. Graczyk and Jakubowski \cite{Graczyk2011} 
(and independently Kuznetsov \cite{Kuz2010}) have discovered a series representation for the logarithm of the Wiener-Hopf factor:
as we have mentioned above, an asymptotic version
of this result was found earlier by Bingham (see formula (9) in \cite{Bingham1973}). 
The first passage time of stable processes was studied in a recent paper by Graczyk and  Jakubowski \cite{Graczyk2011b}.

\begin{figure}
\centering
\subfloat[][$N=10$, $\alpha \in (0,2)$ and $\rho \in (0,1)$]{\label{fig_Vdxds_set1}\includegraphics[height =5.25cm]{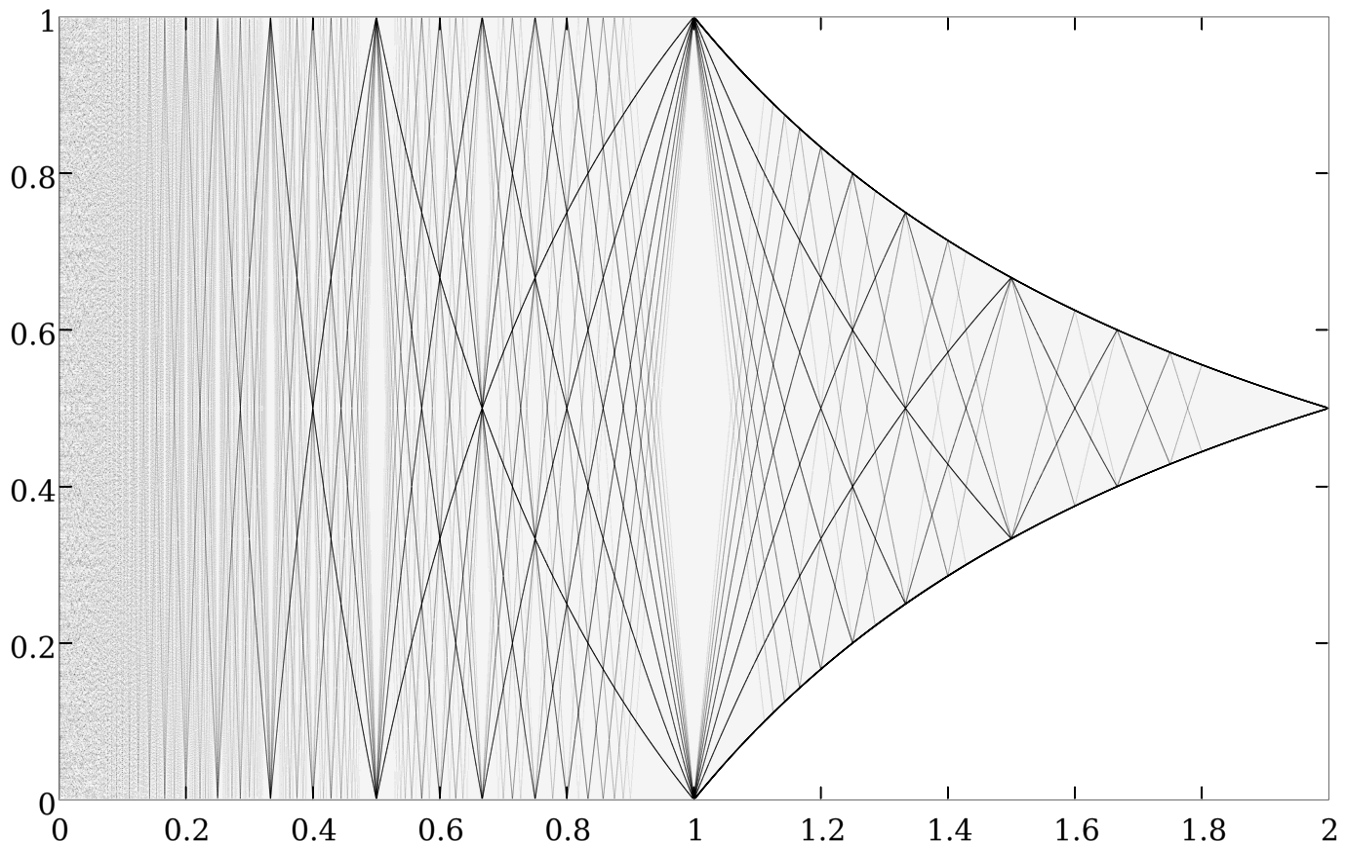}}
\hspace{0.5cm}
\subfloat[][$N=100$, $\alpha \in (0,2)$ and $\rho \in (0,1)$]{\label{fig_Vdxds_set2}\includegraphics[height =5.25cm]{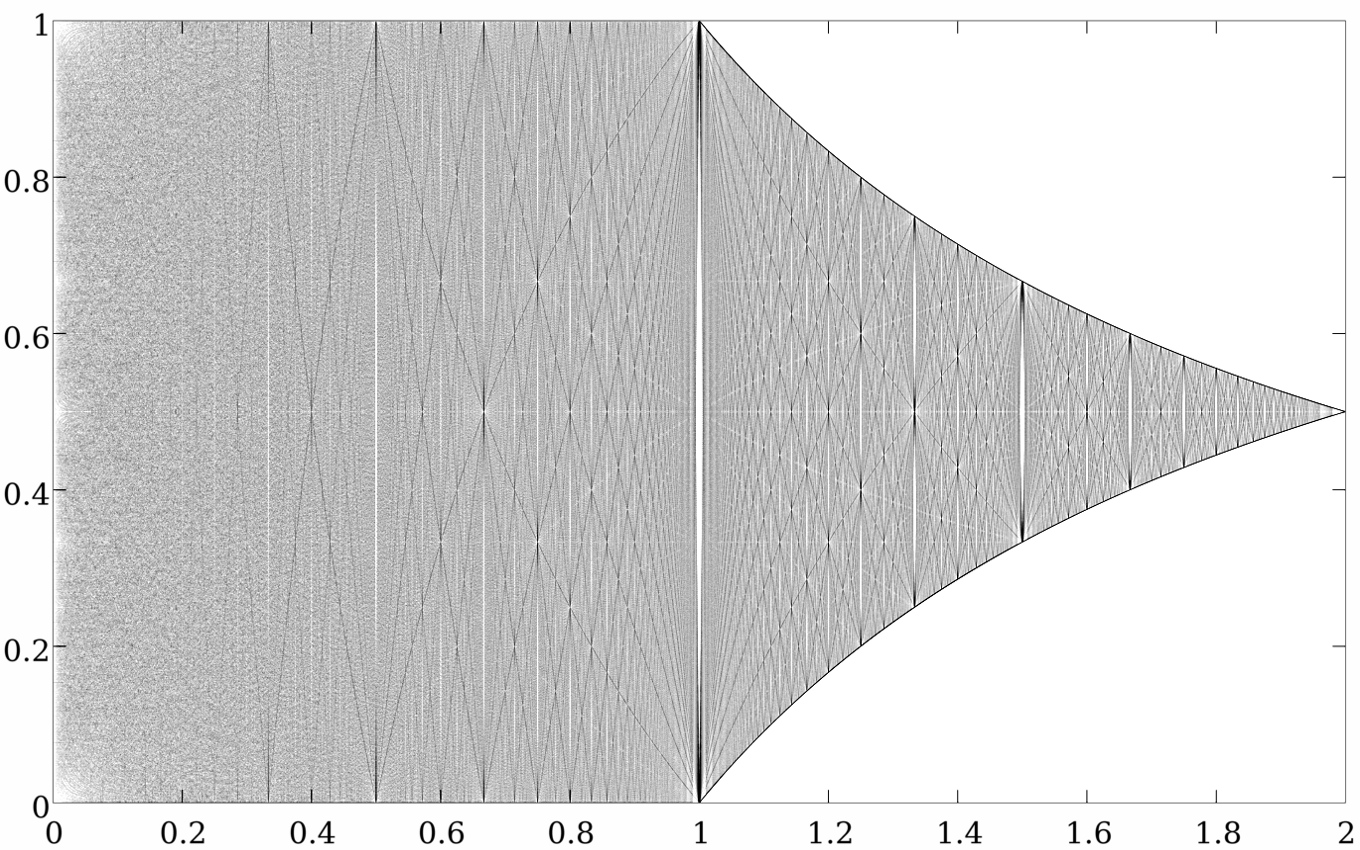}} \\
\subfloat[][$N=100$, $\alpha \in (1/2,1)$ and $\rho \in (1/4,3/4)$]{\label{fig_Vdxds_set3}\includegraphics[height =5.25cm]{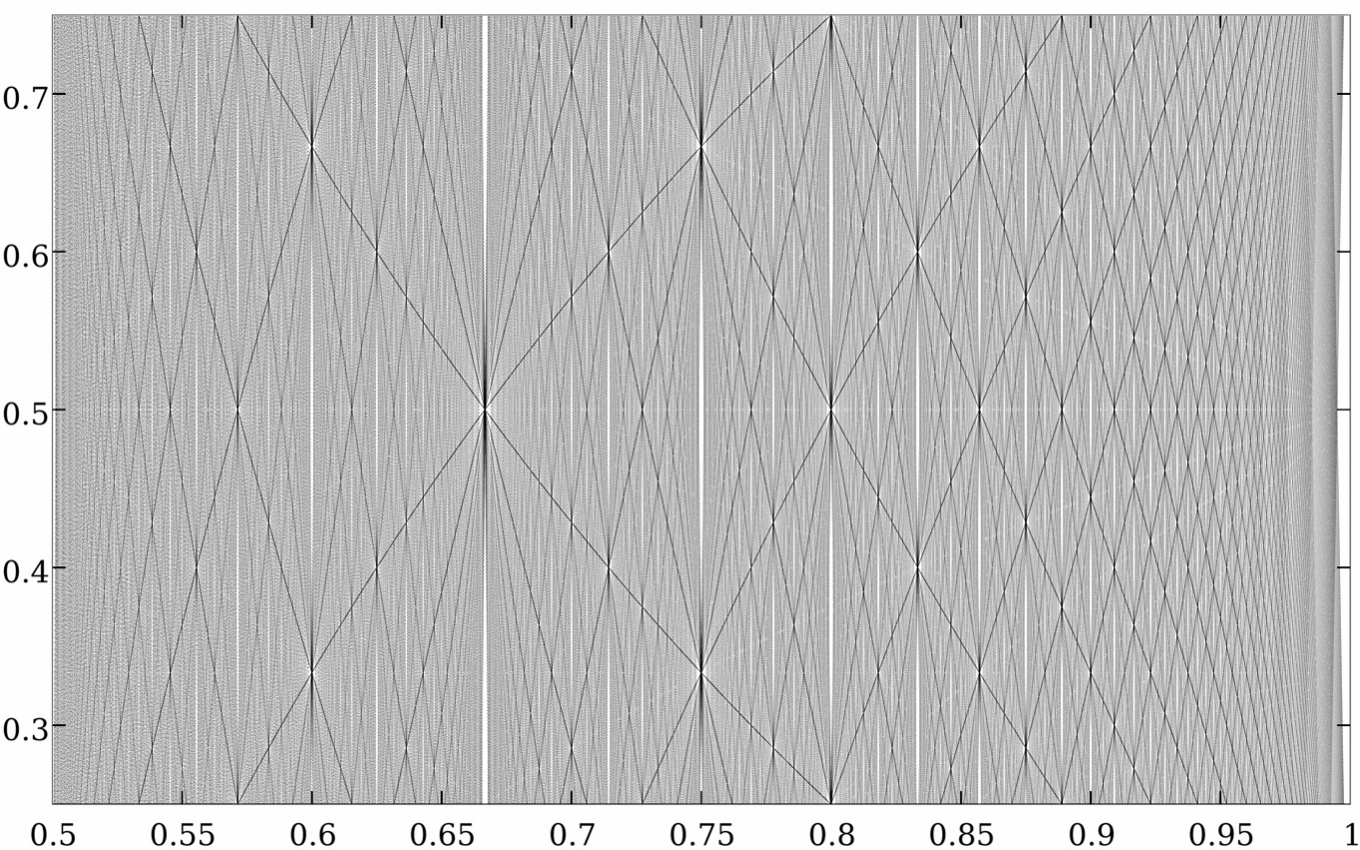}} 
\hspace{0.5cm}
\subfloat[][$N=100$, $\alpha \in (3/2,2)$ and $\rho \in (1/3,2/3)$]{\label{fig_Vdxds_set4}\includegraphics[height =5.25cm]{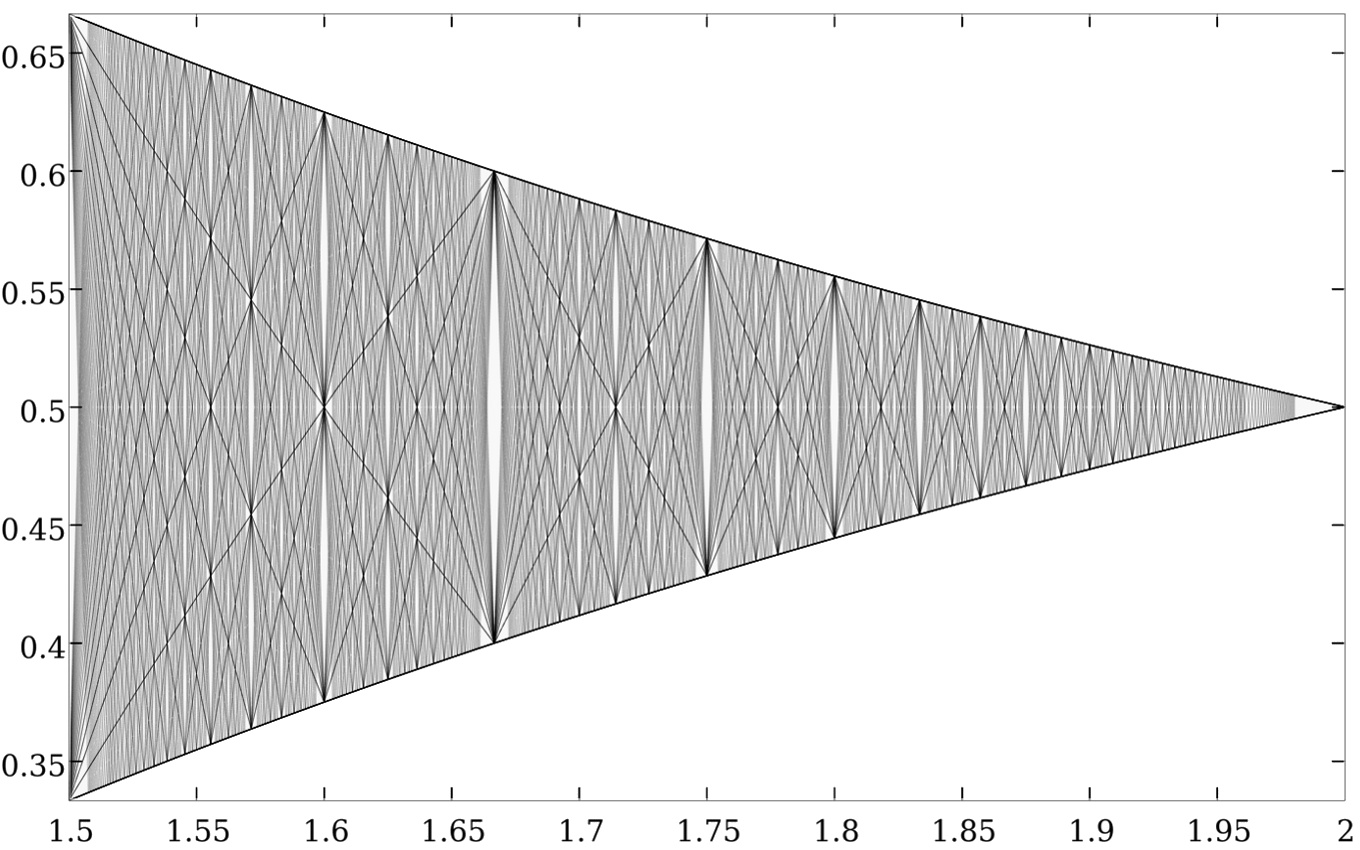}}  
\caption{The set of admissible parameters ${\mathcal A}$ and the curves $\rho=\{ l \alpha^{-1} \}$ for $|l|\le N$.} 
\label{fig1}
\end{figure}

In \cite{Kuz2010} we have followed the approach based on \eqref{piece1} and \eqref{piece2}, and we have obtained
explicit formulas for the Wiener-Hopf factor $\phi(z)$ and for the Mellin transform of $S_1$ in the general case. 
We have also found a complete asymptotic expansion for $p(x)$ as $x\to 0^+$ or $x\to +\infty$, and 
in the case when $X\in C_{k,l}$ we have proved that this asymptotic expansion gives an absolutely convergent series representation
for $p(x)$. The key ingredient which allowed us to derive all these results is the so-called double gamma function
$G(z;\tau)$ (see Section \ref{section_Mellin}), which was introduced and studied by Barnes in \cite{Barnes1899}, \cite{Barnes1901}. 
Hubalek and Kuznetsov in \cite{HubKuz2011} have proved that the asymptotic series
for $p(x)$ obtained in \cite{Kuz2010} in fact provides an absolutely convergent series representation for almost all irrational values of $\alpha$. In this result the arithmetic properties of $\alpha$ are crucial: The result is valid for all irrational $\alpha$ which are not ``too close" to rational numbers. Kuznetsov and Pardo in \cite{KUPA11} have shown that it is also possible to derive some of the results in \cite{Kuz2010} as well as several new results (such as the entrance law of the excursion measure of a stable process reflected at the supremum/infimum) using 
a different technique,  which is not based on \eqref{piece1} and \eqref{piece2}. In some sense this approach is more direct
and more probabilistic in spirit, it's main ingredients are exponential functionals of hypergeometric L\'evy processes and the Lamperti's transformation for positive self-similar Markov processes and on .

In this paper we have three main objectives. First of all, the proof of the convergence of the series for $p(x)$  given in \cite{HubKuz2011} leaves open a possibility that this result is not only true for {\it almost all} irrational $\alpha$, 
but may be valid for {\it all} irrational $\alpha$. As we show in Section \ref{section_convergence}, 
this conjecture is not true: we construct an uncountable dense set of irrational $\alpha$ for which the series does
not converge absolutely for almost all $\rho$. Our second goal is to investigate the case when $\alpha$ is rational: 
in this case the infinite series representation for $p(x)$ given in \cite{Kuz2010} and \cite{HubKuz2011} is not well-defined. 
This case is clearly important for applications, after all when we perform numerical computations we can only work with rational approximations to irrational numbers. In Section \ref{section_Mellin} we will show that when $\alpha$ is rational the Mellin transform of $S_1$ can be given explicitly in terms of rather simple and easily computable functions (Gamma function and dilogarithm). Our third goal is 
to compute the density of $S_1$ numerically. Given the unusual properties of various representations of $p(x)$, 
it is a non-trivial question of whether these results can be used for numerical computations of $p(x)$. 
As we will see in  Section \ref{section_numerics}, the answer to this question is positive, and the 
density of the supremum can be computed quite easily both in the case when $\alpha$ is rational and when $\alpha$ is irrational. Finally, in Section \ref{section_conclusion} 
we discuss the  connections that this problem has to other areas of Mathematics and Mathematical Physics, and
we also suggest several open problems.

%****************************************************************************************************************
%****************************************************************************************************************
%****************************************************************************************************************

\section{Convergence properties of the series representation of $p(x)$}\label{section_convergence}

%****************************************************************************************************************
%****************************************************************************************************************
%****************************************************************************************************************

Let us first introduce several notations and definitions. 
For $x\in \r$ we define the 
``floor" function
 $\lfloor x \rfloor:=\max \{ n\in \zz \; : \; n\le x\}$, the fractional part $\{x\}:=x-\lfloor x \rfloor$ and 
 we let $\|x\|:=\min \{ |x-n| \; : \; n \in \zz\}$ denote the distance to the nearest integer. 
 For any real $\alpha$ we define the set 
\beq\label{def_C_alpha}
{\mathcal C}(\alpha):=\big \{\rho \in (0,1) \; : \; \rho=\{l\alpha^{-1}\}  {\textnormal{ for some }} l \in \zz \big \}.  
\eeq
Thus $\rho \in {\mathcal C}(\alpha)$ if and only if  $X \in {\mathcal C}_{k,l}$ for some integers 
$k$ and $l$. Note that if $\alpha=m/n$ for some co-prime $m$ and $n$, then ${\mathcal C}(\alpha)$ is the finite set 
\beqq
{\mathcal{C}}(m/n)=\{j/m \; : \; j=1,2,\dots,m-1\}.
\eeqq
The sets ${\mathcal{C}}(m/n)$ are clearly visible on Figure \ref{fig1}, they consist of points of intersections of different curves $\rho=\{l/\alpha\}$.  If $\alpha$ is irrational then the set ${\mathcal C}(\alpha)$ is countable and dense in $(0,1)$.

The following set of real numbers was introduced in \cite{Kuz2010} and \cite{HubKuz2011}:
\begin{definition}\label{def_set_L}
Let ${\mathcal L}$ be the set of all real irrational numbers $x$, for which there exists a constant $b>1$ such that the inequality
\beq\label{eqn_def_set_L}
\bigg| x -\frac{p}{q} \bigg| < \frac{1}{b^{q}}
\eeq 
is satisfied for infinitely many integers $p$ and $q$.
\end{definition}

As was proved in \cite{HubKuz2011}, the set ${\mathcal L}$ is rather small: it has Hausdorff dimension zero, which implies that it has Lebesgue measure zero. This set is closed under addition/multiplication by non-zero
rational numbers, therefore it is dense. It can also be conveniently characterized using continued fraction representation of a real number, which is defined as
\beqq	
 x=[a_0;a_1,a_2,\dots]=a_0+\cfrac{1}{a_1+\cfrac{1}{a_2+ \dots }}
 \eeqq
where $a_0 \in {\mathbb Z}$ and $a_i \in {\mathbb N}$ for $i\ge 1$. For $x\notin \q$ the continued fraction has infinitely many terms; truncating it after $n$ steps gives us   a rational number $p_n/q_n=[a_0;a_1,a_2,...,a_n]$, which is called the $n$-th convergent. As was proved in \cite{HubKuz2011}, $x \in {\mathcal L}$ if and only if there exists a constant $b>1$ such that the inequality $a_{n+1}>b^{q_n}$ is satisfied for infinitely many $n$.

For any irrational $\alpha$  we  define sequences $\{a_{m,n}\}_{m\ge 0,n\ge 0}$ and  $\{b_{m,n}\}_{m\ge 0,n\ge 1}$ as
follows
\beq\label{def_a_mn}
a_{m,n}:=\frac{(-1)^{m+n} }{\Gamma\left(1-\rho-n-\frac{m}{\alpha}\right)\Gamma(\alpha\rho+m+\alpha n)}
\prod\limits_{j=1}^{m} \frac{\sin\left(\frac{\pi}{\alpha} \left( \alpha \rho+ j-1 \right)\right)} {\sin\left(\frac{\pi j}{\alpha} \right)} 
\prod\limits_{j=1}^{n} \frac{\sin(\pi \alpha (\rho+j-1))}{\sin(\pi \alpha j)},
\eeq
\beq\label{def_b_mn}
b_{m,n}:=\frac{\Gamma\left(1-\rho-n-\frac{m}{\alpha}\right)\Gamma(\alpha\rho+m+\alpha n) }{\Gamma\left(1+n+\frac{m}{\alpha}\right)\Gamma(-m-\alpha n)}
a_{m,n}.
\eeq

The following Theorem was proved in 
\cite{HubKuz2011} (the statement related to the asymptotic behavior of $p(x)$ has appeared earlier in \cite{Kuz2010}). 
\begin{theorem}\label{thm_main}
Assume that $\alpha \notin {\mathcal L} \cup \q$. Then for all $x>0$ 
\beq\label{eqn_p1}
p(x)& = &  x^{-1-\alpha } \sum\limits_{n\ge 0} \sum\limits_{m\ge 0}b_{m,n+1} x^{-m-\alpha n}, \;\; {\textnormal { if }} \alpha \in (0,1), \\
\label{eqn_p2}
p(x)& = &  x^{\alpha\rho-1} \sum\limits_{n\ge 0} \sum\limits_{m\ge 0} a_{m,n} x^{ m+\alpha n}, \;\;\;\;\;\;\;  {\textnormal { if }} \alpha \in (1,2).
\eeq
The series converges absolutely and uniformly on compact subsets of $(0,\infty)$.  
Moreover, for every $\alpha \notin \q$ the series \eqref{eqn_p1} \{ \eqref{eqn_p2} \} provides complete asymptotic expansion as $x\to +\infty$ \{resp. $x\to 0^+$\}.
\end{theorem}

The above theorem, if suitably interpreted, is also true if $\alpha \in {\mathcal L} \cup \q$ and $\rho \in {\mathcal C}(\alpha)$. In this 
case $X \in {\mathcal C}_{k,l}$ for some integers $k$ and $l$, and it turns out that expressions 
\eqref{def_a_mn}, \eqref{def_b_mn} for
coefficients $a_{m,n}$ and $b_{m,n}$ can be simplified, and the resulting formulas 
are well-defined even for rational values of $\alpha$. See Theorem  10 in \cite{Kuz2010} for all the details.

If we are only interested in the asymptotic expansion of $p(x)$, then 
Theorem \ref{thm_main} and Theorem  10 in \cite{Kuz2010} give us a complete picture. 
If $\alpha \notin \q$  \{ $\alpha \in \q$ and $\rho \in {\mathcal C}(\alpha)$ \} we have explicit asymptotic expansion
due to Theorem \ref{thm_main}  \{ resp. Theorem  10 in \cite{Kuz2010} \}. At the same time, in the remaining case
when $\alpha \in \q$ and $\rho\notin {\mathcal C}(\alpha)$ we know that the asymptotic expansion will contain some logarithmic terms coming from the multiple poles of the Mellin transform of $S_1$ (see the proof of Theorem 9 in \cite{Kuz2010}). 
Note that even if $\alpha\in \q$ and $\rho\notin {\mathcal C}(\alpha)$ the first terms of the series \eqref{eqn_p1}, \eqref{eqn_p2} which are well-defined still provide the first asymptotic terms for $p(x)$: this follows easily from the proof of Theorem 9 in \cite{Kuz2010}.

The situation with the absolute convergence of the series \eqref{eqn_p1}, \eqref{eqn_p2} is much less clear. 
Let us summarize the present state of knowledge. 
\begin{itemize}
\item[(i)] If $\alpha \in \q$ then the series \eqref{eqn_p1}, \eqref{eqn_p2} are well-defined and converge only for a finite set of $\rho \in {\mathcal C}(\alpha)$. This case corresponds to the points of intersections of black curves and to ``white vertical intervals" on Figure
\ref{fig1}.  
\item[(ii)] If $\alpha \notin \q$ then the series \eqref{eqn_p1}, \eqref{eqn_p2} converge for a dense countable set of $\rho \in {\mathcal C}(\alpha)$. 
\item[(iii)] If $\alpha \notin \q$ and $\alpha \notin \ll$ then the series \eqref{eqn_p1}, \eqref{eqn_p2} converge for all $\rho \in (0,1)$. 
\end{itemize}

\noindent
We see that the only remaining case for which we don't know anything about the convergence of the series \eqref{eqn_p1}, \eqref{eqn_p2} 
is when $\alpha \in {\mathcal L}$. The most aesthetically pleasing and natural conjecture would be that 
condition $\alpha \notin \ll$ in (iii) can be dropped: 
\begin{center}
{\bf { Conjecture: if $\alpha \notin \q$ then the series \eqref{eqn_p1}, \eqref{eqn_p2} converge for all $\rho \in (0,1)$. }}
\end{center}
Our main result in this section is that this conjecture is in fact {\bf false}. We will show that by slightly modifying the definition of the set $\ll$ we can construct an uncountable dense subset $\tl \subset \ll$ such that for every $\alpha \in \tl$ the series \eqref{eqn_p1}, \eqref{eqn_p2} does not converge absolutely for almost all $\rho \in (0,1)$. This set $\tl$ is defined as follows.

\begin{definition}\label{def_set_tl}
 Let $\tl$ be the set of all real irrational numbers $x$, for which there exist $\epsilon>0$ and $b>1$ such that the inequality
\beq\label{eqn_def_set_L}
 \bigg| x -\frac{p}{q} \bigg| < b^{-q \ln(q)^{1+\epsilon}}
\eeq 
is satisfied for infinitely many integers $p$ and $q$.
\end{definition}

As the next Proposition shows, the structure of the set ${\mathcal L}$ can also be conveniently described in terms of the continued fraction representation of real numbers. The proof of this Proposition is omitted, as it 
is essentially identical to the proof of Proposition 1 in \cite{HubKuz2011}.
\begin{proposition}\label{prop_set_L}

${}$
 \begin{itemize}
  \item[(i)] If $x\in \tl$ then $z x \in \tl$ and $z+x \in \tl$ for all  $z \in {\mathbb Q}\setminus\{0\}$.
  \item[(ii)]  $x \in \tl$ if and only if $x^{-1} \in \tl$.
  \item[(iii)] Let $x=[a_0;a_1,a_2,\dots]$. Then $x \in \tl$ if and only if there exist $\epsilon>0$ and $b>1$ such that the inequality
$$a_{n+1}>b^{q_n \ln(q_n)^{1+\epsilon}}$$ is satisfied for infinitely many $n$. 
 \end{itemize}
\end{proposition}

We see from Proposition \ref{prop_set_L} that the set $\tl$ is an uncountable set, it is closed
under addition and multiplication by non-zero rational numbers, in particular it is a dense set. 
It is of course a rather small set: being a subset of ${\mathcal L}$ it also has Hausdorff dimension zero and, therefore, Lebesgue measure zero.

The next Theorem is our main result in this section. 
\begin{theorem}\label{main_theorem} 
For every $\alpha \in \tl$
there exists a set of real numbers ${\mathcal B}(\alpha)$ having zero Hausdorff dimension 
(and, therefore,  zero Lebesgue measure) and such that for every $\rho \notin {\mathcal B}(\alpha)$ and every $x>0$ the series
\eqref{eqn_p1}, \eqref{eqn_p2} do not converge absolutely. 
\end{theorem}
\begin{proof}
Assume that $\alpha \in (1,2)\cup \tl$. According to Definition \ref{def_set_tl} there exist $b>1$ and $\epsilon>0$ such that 
\beq\label{est0}
\| \alpha q \|< q b^{-q \ln(q)^{1+\epsilon}}, \;\; {\textnormal{ for all }} q \in {\mathcal Q}, 
\eeq
where ${\mathcal Q}$ is an infinite subset of $\n$.  
 We define the following set of real numbers 
\beq\label{def_B_alpha}
{\mathcal B}(\alpha):=  \{\rho \in (0,1) \; : \; 
 \| \alpha \rho + n \alpha\|<n^{- \ln \ln (1+n)} {\textnormal{ for infinitely many }} \; n\in \n \} .
\eeq
From the definition \eqref{def_C_alpha} of the set ${\mathcal C}(\alpha)$ and property 
\eqref{est0} it is clear that ${\mathcal C}(\alpha)\subseteq {\mathcal B}(\alpha)$, in particular ${\mathcal B}(\alpha)$ is non-empty. Our first goal is to prove that the set ${\mathcal B}(\alpha)$ is in fact quite small: the Hausdorff dimension of this set is zero.

For $v>0$ let us define 
\beqq
{\mathcal V}_v(\alpha):=\{x \in \r \; : \; \| x + n \alpha\|<n^{- v} {\textnormal{ for infinitely many }} \; n\in \n \}. 
\eeqq 
As was proved by Bugeaud (see Theorem 1 in \cite{Bugeaud2003}), for any irrational $\alpha$ and any $v>1$ the Hausdorff dimension of the set 
${\mathcal V}_v(\alpha)$ is $1/v$. From \eqref{def_B_alpha} we find that for any $v>0$
\beqq
\{\alpha x \; : \; x\in {\mathcal B}(\alpha)\} \subset {\mathcal V}_v(\alpha),
\eeqq
and since $v$ can be taken arbitrary large we see that the Hausdorff dimension of ${\mathcal B}(\alpha)$ must be zero.  

Next, we will show that for any $\rho \notin {\mathcal B}(\alpha)$ than for any $x>0$ we have
\beq
|a_{0,q} x^{q}|>1 \;\;\; \textnormal{ for  all large enough $q\in {\mathcal Q}$,}
\eeq
which of course implies that the series \eqref{eqn_p2} can not converge absolutely. 

First of all, from \eqref{def_a_mn} we find that for $n\ge 1$
\beq\label{estimate_a0n_N1}
|a_{0,n}|&=&\frac{\sin(\pi \rho)}{\pi} \frac{\Gamma(\rho+n)}{\Gamma(\alpha \rho + \alpha n)}
\prod\limits_{j=1}^{n} |\sin(\pi \alpha (\rho+j-1))|
\prod\limits_{i=1}^{n} \frac{1}
{|\sin(\pi \alpha i)|}\\ \nonumber
&>& \frac{\sin(\pi \rho)}{\pi} \frac{1}{(2n)!} \left(\prod\limits_{j=1}^{n} \| \alpha (\rho+j-1))\|\right)
\frac{1}{\pi \| \alpha n \|},
\eeq
where the first equality is obtained by applying the reflection formula for the Gamma function, 
and the inequality in the second line
is derived using the following trivial estimates: 
\beqq
&&\Gamma(\rho + n)>\Gamma(1)=1, \\
&&\Gamma(\alpha\rho+\alpha n)<\Gamma(1+2n)=(2n)!, \\ 
&&|\sin(\pi z)|\ge\|z\|, \; |\sin(\pi z)|\le 1  \; \textnormal{ and } \; |\sin(\pi z)|\le \pi \| z\|, \;\;\; z\in \r. 
\eeqq 

Next, since $\rho \notin {\mathcal B}(\alpha)$ we know that $\rho \notin {\mathcal C}(\alpha)$, in 
particular $\|\alpha \rho\|>0$. At the same time, since $\rho \notin {\mathcal B}(\alpha)$ there exists 
a constant $C=C(\alpha,\rho)>0$ such that for all $j\ge 1$
\beq\label{est1}
\|\alpha \rho + j \alpha \|>C j^{-\ln \ln (1+j)}. 
\eeq
Using \eqref{est1} 
and the estimate
\beqq
\sum\limits_{j=1}^{n-1} \ln(j) \ln \ln (1+j) <n \ln (n) \ln\ln(1+n)
\eeqq
we find
\beq\label{est3}
\prod\limits_{j=1}^{n} \| \alpha (\rho+j-1))\|&=&
\|\alpha \rho \| \prod\limits_{j=1}^{n-1} \| \alpha \rho+ \alpha j\| \\ \nonumber
&>& \|\alpha \rho \| C^{n-1} \exp\left( - \sum\limits_{j=1}^{n-1} \ln(j) \ln \ln (1+j) \right)
> \|\alpha \rho\| C^{n-1} e^{-n \ln(n) \ln\ln(1+n)}.
\eeq
Next, we take $q\in {\mathcal Q}$, combine  \eqref{est0}, \eqref{estimate_a0n_N1}, \eqref{est3} with the trivial estimate $(2n)!<(2n)^{2n}$ and find that for every $x>0$ 
\beqq
|a_{0,q}x^q|&>&\frac{\sin(\pi \rho)\|\alpha \rho\|}{\pi^2} C^{q-1} x^q e^{-2q \ln(2q)-q \ln(q) \ln \ln(1+q)} \frac{1}{\|\alpha q\|}\\ 
&>&
\frac{\sin(\pi \rho)\|\alpha \rho\|}{\pi^2} C^{q-1} x^q e^{-2q \ln(2q)-q \ln(q) \ln \ln(1+q)}  b^{q\ln(q)^{1+\epsilon}} q^{-1}.  
\eeqq
The right-hand side of the above inequality can become arbitrarily large if $q$ is large enough. 
Given that ${\mathcal Q}$ is an infinite set we conclude that that for any $x>0$ it is true that $|a_{0,n}x^n|>1$ for infinitely many $n$, which shows that the series  \eqref{eqn_p2} does not converge absolutely. 

The proof in the case $\alpha \in (0,1)$ is very similar, the details are left to the reader.
\end{proof}

%****************************************************************************************************************
%****************************************************************************************************************
%****************************************************************************************************************

\section{Mellin transform of $S_1$ when $\alpha$ is rational}\label{section_Mellin}

%****************************************************************************************************************
%****************************************************************************************************************
%****************************************************************************************************************

The fact that the Mellin transform is a powerful tool for studying stable processes was known already in 1950s  (see Zolotarev \cite{Zolotarev1957} and Darling \cite{Darling1956}). This is not a very surprising fact, given that 
the scaling property for stable processes is multiplicative in nature. We will use the following notation for the Mellin transform of $S_1$
\beq\label{def_mms}
\mm(w)=\mm(w;\alpha,\rho):=\e \left[ \left(S_1\right)^{w-1} \right]=\int\limits_{0}^{\infty} p(x;\alpha,\rho) x^{w-1} \d x. 
\eeq
The analytical structure of this function was completely described in \cite{Kuz2010}. In particular, 
for all admissible parameters $(\alpha,\rho)$ there exists an explicit expression
\beq\label{formula_Ms_general}
\mm(s)=\alpha^{s-1} \frac{G(\alpha\rho;\alpha)}{G(\alpha(1-\rho)+1;\alpha)}
\times \frac{G(\alpha(1-\rho)+2-s;\alpha)}
{G(\alpha\rho-1+s;\alpha)}
\times \frac{G(\alpha-1+s;\alpha)}{G(\alpha+1-s;\alpha)}.
\eeq
Here $G(z;\tau)$ is the so-called double gamma function, which was studied by Barnes in \cite{Barnes1899}, \cite{Barnes1901}. This function is defined for $|\arg(\tau)|<\pi$, and all $z\in \c$ by a Weierstrass product 
\beqq
G(z;\tau):=\frac{z}{\tau} e^{a\frac{z}{\tau}+b\frac{z^2}{2\tau}} \prod\limits_{m\ge 0} \prod\limits_{\substack{n\ge 0 \\ m+n>0 }} 
\left(1+\frac{z}{m\tau+n} \right)e^{-\frac{z}{m\tau+n}+\frac{z^2}{2(m\tau+n)^2}}.
\eeqq
In particular we see that $G(z;\tau)$ is an entire function of $z$, which has zeros at points $-m\tau - n$, $n,m\ge 0$. 
Barnes \cite{Barnes1899} has proved that it is possible to choose the coefficients $a=a(\tau)$ and $b=b(\tau)$ so that  $G(1;\tau)=1$ and $G(z;\tau)$ satisfies the two functional equations
\beq\label{funct_rel_G}
G(z+1;\tau)=\Gamma\left(\frac{z}{\tau}\right) G(z;\tau), \;\;\; G(z+\tau;\tau)=(2\pi)^{\frac{\tau-1}2}\tau^{-z+\frac12}
\Gamma(z) G(z;\tau).
\eeq

 Formula \eqref{formula_Ms_general} can be considerably simplified when $X\in {\mathcal C}_{k,l}$, in this case $\mm(s)$ 
 can be given in terms of Gamma function and trigonometric functions (see formula (6.10) in \cite{Kuz2010}). As we will see in this section, formula \eqref{formula_Ms_general} can also be simplified in the case when 
$\alpha$ is rational.  In order to state our results, first we need to present several definitions.

Everywhere in this paper we will work with the principal branch of the logarithm, which is defined in the domain $|\arg(z)|<\pi$ 
by requiring that $\ln(1)=0$. Similarly, the power function will be defined as $z^a=\exp(a\ln(z))$ in the domain $|\arg(z)|<\pi$. 
The principal branch of the dilogarithm function $\Li_2(z)$ is defined in the domain $z \in \c \setminus [1,\infty)$ by the integral representation
\beq\label{def_Li2}
\Li_2(z):=-\int\limits_0^z \frac{\ln(1-u)}{u} \d u,
\eeq
where we integrate along any path from $0$ to $z$ which lies in $\c \setminus [1,\infty)$
(see \cite{Lewin} and \cite{Maximon2003}). By expanding the logarithm into Taylor series it is easy to see that in the open disk $\{z \in \c \; : \; |z|<1\}$ the dilogarithm is given by an absolutely convergent series
\beq\label{Li_series1}
\Li_2(z)=\sum\limits_{k\ge 1} \frac{z^2}{k^2}.
\eeq
While dilogarithm is not one of the elementary functions, it is a well-understood function which satisfies many functional identities. As we will see later in Section \ref{section_numerics}, the dilogarithm can be easily computed to high precision in its entire domain of definition $\c \setminus [1,\infty)$.

For $a, q  \in \c$ such that $|a|<1$ and $|q|\le 1$  we define the modified q-Pochhammer symbol
\beq\label{def_gen_q_pochhammer}
[a;q]_n:=\prod\limits_{k=1}^{n-1} (1-aq^k)^\frac{k}{n},
\eeq
and for any co-prime positive integers $m$ and $n$ we define 
\beq\label{def_Hmn}
H_{m,n}(s):=\exp\left(-\frac{1}{2 \pi \i mn} {\textnormal{Li}}_2\left(e^{2\pi \i s}\right)\right)
\times \frac{\left(1-e^{2\pi \i s} \right)^{1-\frac{s}{mn}}}
{\left[ e^{\frac{2\pi \i s}{m}};e^{\frac{2\pi \i n}{m}} \right]_m 
\left[ e^{\frac{2\pi \i s}{n}};e^{\frac{2\pi \i m}{n}} \right]_n}.
\eeq
Note that $H_{m,n}(s)$ is well-defined and is an analytic function in the half-plane $\im(s)>0$. The following Theorem is our main result in this Section. 
\begin{theorem}\label{thm_Mellin_rational} Assume that $\alpha=m/n$ where $m$ and $n$ are co-prime natural numbers. Then
 for all $s\in \c$ with $\im(s)>0$ we have
\beq\label{eqn_ms_main}
\mm(s)=\sqrt{\frac{n}{m}}e^{\frac{\pi \i }{12mn}(m^2+n^2-3mn)+ \pi \i \left( \frac{n}{m}-\rho \right) (s-1) }
\frac{\Gamma(s)}{\Gamma\left(1-\frac{n}{m}(1-s)\right)}
\frac{H_{m,n}(m\rho)H_{m,n}(ns)}{H_{m,n}(n(s-1)+m\rho)}. 
\eeq
\end{theorem}
\begin{proof}
We start with the following identity
\beqq
G(\alpha-1+s;\alpha)=(2\pi)^{\frac{\alpha-1}{2}} \alpha^{-s+\frac{1}{2}} \frac{\Gamma(s)}{\Gamma\left(1+\frac{s-1}{\alpha}\right)} G(s;\alpha),
\eeqq
which follows from \eqref{funct_rel_G}.
The above identity and formula \eqref{formula_Ms_general} give us
\beq\label{formula_Ms_general2}
\mm(s)=(2\pi)^{\frac{\alpha-1}{2}} \alpha^{-\frac12} 
\frac{\Gamma(s)}{\Gamma\left(1-\frac{1-s}{\alpha}\right)}
\frac{G(\alpha\rho;\alpha)}{G(\alpha(1-\rho)+1;\alpha)}
 \frac{G(\alpha(1-\rho)+2-s;\alpha)}
{G(\alpha\rho-1+s;\alpha)}
\frac{G(s;\alpha)}{G(\alpha+1-s;\alpha)}.
\eeq
The first main ingredient in our proof is the following formula
\beq\label{G_finite_product}
G\left(z;\frac{m}{n} \right)=(2\pi)^{-(n-1) \frac{z}{2}} n^{\frac{n}{m}\frac{z^2}{2}-(m+n) \frac{z}{2m}+1}
\prod\limits_{\substack{0\le k \le m-1 \\ 0\le l \le n-1}} G\left(\frac{z+k}{m}+\frac{l}{n} \right)
\prod\limits_{\substack{0\le k \le m-1 \\ 0\le l \le n-1 \\ k+l>0}}G\left(\frac{k}{m}+\frac{l}{n} \right)^{-1}
\eeq
where $G(z):=G(z;1)$ is Barnes G-function. The above formula is just a special case of a more general result  connecting $G(z;\frac{m}{n}\tau)$ with $G(z;\tau)$, which was established
by Barnes in \cite{Barnes1901}. In order to obtain the identity \eqref{G_finite_product} one has to set $\tau=w_1=w_2=1$ in the formulas on pages 302 and 359 in \cite{Barnes1901}. Applying 
identity \eqref{G_finite_product} to each double gamma function in \eqref{formula_Ms_general2} and simplifying 
the resulting expression we obtain
\beq\label{eqn_Ms_n3}
\mm(s)=(2\pi)^{\frac{m-n}2} \alpha^{-\frac{1}{2}}  
\frac{\Gamma(s)}{\Gamma\left(1-\frac{1-s}{\alpha}\right)}
\frac{F_{m,n}(m\rho)  F_{m,n}(ns)}{F_{m,n}(n(s-1)+m\rho)},
\eeq
where we have defined
\beq\label{def_Fmn}
F_{m,n}(s):=
\prod\limits_{k=0}^{m-1} \prod\limits_{l=0}^{n-1} 
\frac{G\left(\frac{s}{nm}+\frac{k}{m}+\frac{l}{n}\right)}{G\left(\frac{1+k}{m}+\frac{1+l}{n}-\frac{s}{nm}\right)}
\eeq
We change the indices $k \mapsto m-1-k$ and $l \mapsto n-1-l$ in the denominator in the right-hand side 
of \eqref{def_Fmn} and obtain an equivalent expression
\beq\label{eqn_Fmn_2}
F_{m,n}(s)=
\prod\limits_{k=0}^{m-1} \prod\limits_{l=0}^{n-1} 
\frac{G\left(\frac{s}{nm}+\frac{k}{m}+\frac{l}{n}\right)}{G\left(2-\frac{s}{nm}-\frac{k}{m}-\frac{l}{n}\right)}.
\eeq
The second main ingredient of the proof is the following reflection formula for the Barnes G-function 
\beq\label{reflection_formula}
\frac{G(z)}{G(2-z)}= (2\pi)^{z-1} \left(1-e^{2\pi \i z} \right)^{1-z} 
e^{\frac{\pi \i}2 \left( z^2 -2z+\frac{5}{6} \right)
-\frac{1}{2\pi \i} {\textnormal{Li}}_2(\exp(2\pi \i z)) }, \;\;\; \im(z)>0.  
\eeq
The above formula follows from identity (A.5) in \cite{Voros} and  formulas (4.6), (8.43) in \cite{Lewin}. 
Next, we apply the reflection formula \eqref{reflection_formula} to each term in the product \eqref{eqn_Fmn_2}, we simplify the resulting expression with the help of identities 
\beqq
\sum\limits_{k=1}^{n-1} k =\frac{1}{2} n (n-1), \;\;\;  \sum\limits_{k=1}^{n-1} k^2 =\frac{1}{6} n (n-1) (2n-1),
\eeqq
and we finally obtain
\beq\label{eqn_Fmn3} 
\nonumber
F_{m,n}(s)&=&
(2\pi)^{s-\frac{m+n}2} e^{\frac{\pi \i s}{2mn} (s-m-n)+\frac{\pi \i }{12mn}\left(m^2+n^2+3mn \right)} \\
&\times &
\prod\limits_{k=0}^{m-1} \prod\limits_{l=0}^{n-1} 
\left( 1-\exp\left(2\pi \i \left(\frac{s}{mn}+\frac{k}{m}+\frac{l}{n} \right)\right) \right)^{1-\frac{s}{mn}+\frac{k}{m}+\frac{l}{n}} \\ \nonumber
&\times& \prod\limits_{k=0}^{m-1} \prod\limits_{l=0}^{n-1}  e^{-\frac{1}{2\pi \i} \Li_2\left(\exp\left(2\pi \i \left(\frac{s}{mn}+\frac{k}{m}+\frac{l}{n} \right)\right) \right) }.
\eeq
Formula \eqref{eqn_Fmn3} can be further simplfied with the help of the following two identities, which are valid for any co-prime positive integers $m$ and $n$:
\beq\label{eqn_id_1}
\prod\limits_{k=0}^{n-1} \left(1-z e^{\frac{2\pi \i k m}{n}} \right)=1-z^n, 
\;\;\;\;\; \sum\limits_{k=0}^{n-1} \Li_2\left(ze^{\frac{2\pi \i k m}{n}} \right) = \frac{1}{n} \Li_2(z^n), \;\;\; |z|<1.
\eeq
The first identity can be easily verified if one considers the left-hand side as a product involving $n$ roots of the polynomial $1-z^n$. For the second identity see (3.12) in \cite{Maximon2003}. 

Using \eqref{eqn_Fmn3}, \eqref{def_gen_q_pochhammer} and  \eqref{eqn_id_1} we can express the function $F_{m,n}(s)$ as follows
\beq\label{eqn_Fmn4} 
F_{m,n}(s)=(2\pi)^{s-\frac{m+n}2} e^{\frac{\pi \i s}{2mn} (s-m-n)+\frac{\pi \i }{12mn}\left(m^2+n^2+3mn \right)
-\frac{1}{2 \pi \i mn} {\textnormal{Li}}_2(\exp(2\pi \i s))}
\times \frac{\left(1-e^{2\pi \i s} \right)^{1-\frac{s}{mn}}}
{\left[ e^{\frac{2\pi \i s}{m}};e^{\frac{2\pi \i n}{m}} \right]_m 
\left[ e^{\frac{2\pi \i s}{n}};e^{\frac{2\pi \i m}{n}} \right]_n}.
\eeq
Formula \eqref{eqn_ms_main} now follows from \eqref{def_Hmn} and \eqref{eqn_Ms_n3}, \eqref{eqn_Fmn4}. 
\end{proof}

Theorem \ref{thm_Mellin_rational} should be compared with Theorem 2 in \cite{Kuz2010}, which gives an 
explicit expression for the Wiener-Hopf factor $\phi(z)$ in terms of the Clausen function, defined
for real values of $x$ as 
$\textnormal{Cl}_2(x):=\im \left[ \Li_2(\exp(\i x) \right]$ (see chapter 4 in \cite{Lewin}). This shows that
when $\alpha$ is rational both the Wiener-Hopf factor
$\phi(z)$ and the Mellin transform $\mm(s)$ can be expressed in terms of elementary functions, Gamma function and dilogarithm. 
It would be an interesting exercise to obtain expression  \eqref{eqn_ms_main} for the Mellin transform directly from Theorem 2 in \cite{Kuz2010} using identity \eqref{piece1}, or, alternatively, using the approach based on exponential functionals and Lamperti's transformation  (see \cite{KUPA11}).

%****************************************************************************************************************
%****************************************************************************************************************
%****************************************************************************************************************

\section{Numerical experiments}\label{section_numerics}

%****************************************************************************************************************
%****************************************************************************************************************
%****************************************************************************************************************

In this section we study the problem of numerically computing the density of the supremum $S_1$. 
There are two different approaches to this problem depending on whether $\alpha$ is rational or not. In the case when 
$\alpha$ is rational we can compute $p(x)$ via the inverse Mellin transform of $\mm(s)$ given by \eqref{eqn_ms_main}. 
If $\alpha$ is irrational, the Mellin transform $\mm(s)$ given by
the general formula \eqref{eqn_ms_main} is rather hard to evaluate numerically, thus we will try to compute $p(x)$ using the series expansions  \eqref{eqn_p1} or \eqref{eqn_p2}. However, apriori it is not clear 
whether these expansions can possibly lead to any meaningful numerical results. According to Theorem \ref{thm_main} we need to ensure that $\alpha \notin {\mathcal L} \cup \q$, but in any computer program
 $\alpha$ would be given to a finite precision. Therefore,  we will be working with a rational approximation to $\alpha$, and we know that
 for rational values of $\alpha$ the coefficients of the series \eqref{eqn_p1} and \eqref{eqn_p2} are not even well-defined. 
 As we will see in this Section, this method based on series expansions can still be used for computing $p(x)$, though it certainly has its limitations.

%****************************************************************************************************************
%****************************************************************************************************************
%****************************************************************************************************************

\subsection{Computing $p(x)$ by inverting the Mellin transform for rational $\alpha$}\label{subsection_numerics_Mellin}

%****************************************************************************************************************
%****************************************************************************************************************
%****************************************************************************************************************

In the method based on inverting the Mellin transform we face the following two problems. First, we need to be able to evaluate the Mellin transform $\mm(s)$ given by \eqref{eqn_ms_main} efficiently and with high  precision.  Second,  we have to compute the inverse Mellin transform numerically. 

Let us describe how one can compute the Mellin transform $\mm(s)$. 
We see that the only challenging part in \eqref{eqn_ms_main} is computing the dilogarithm function $\Li_2(z)$ in the unit disk 
${\mathbb D}:=\{z \in \c \; : \; |z|\le 1 \}$. 
Let us denote
\beqq
D_1:={\mathbb D} \cap \{z \in \c \; : \; |1-z|<1/10\}.
\eeqq

In the domain $D_1$ the function $\Li_2(z)$ can be computed using the identity (3.2) in \cite{Maximon2003}
\beq\label{Li2_series2}
\Li_2(z)=-\Li_2(1-z) + \frac{\pi^2}6 - \ln(1-z) \ln(z).
\eeq
The term $\Li_2(1-z)$ in the right-hand side of \eqref{Li2_series2} is evaluated using the series expansion \eqref{Li_series1}; since $|1-z|<1/10$ in $D_1$ this series converges very fast. 

In the domain ${\mathbb D} \setminus D_1$ we will use the following formula
\beq\label{Li2_series3}
\Li_2(z)=-3w-\frac{w^2}{4}+
2\pi \i \ln\left( \frac{2\pi \i+w}{2\pi \i-w} \right)+2w \sum\limits_{n\ge 1} (-1)^n \frac{\zeta(2n)-1}{2n+1} 
\left(\frac{w}{2\pi} \right)^{2n}, \;\;\; w=\ln(1-z),
\eeq
where $\zeta(2n)-1=2^{-2n}+3^{-2n}+\dots$. 
Formula \eqref{Li2_series3} follows by combining (4.3) in \cite{Maximon2003} and the Taylor series of  $\ln((2\pi \i+w)/(2\pi \i-w))$.  Let us investigate the convergence rate of \eqref{Li2_series3}. First of all, one can check that for all $n\ge 1$ we have
\beqq
2^{2n}(\zeta(2n)-1)<4(\zeta(2)-1)<3.
\eeqq
Next, for all $z \in {\mathbb D} \setminus D_1$ we have $1/10 \le |1-z|\le 2$ and $|\arg(1-z)|< \pi/2$, which implies that 
\beqq
|w|^2=|\ln(1-z)|^2=\ln |1-z|^2 + (\arg(1-z))^2<\ln(10)^2+\pi^2/4.
\eeqq
Therefore, the $n$-th term in \eqref{Li2_series3} can be bounded from above as follows
\beqq
 \frac{\zeta(2n)-1}{2n+1} 
\left(\frac{|w|}{2\pi} \right)^{2n}
&<& 3 \times  2^{-2n}\left(\frac{|w|}{2\pi}\right)^{2n}<
3 \times\left(\frac{\ln(10)^2+\pi^2/4}{16\pi^2}\right)^n
\approx  3 \times(0.04919...)^n < \frac{3}{20^n}.
\eeqq
This shows that the series \eqref{Li2_series3} also converges very fast.

\begin{figure}
\centering
\subfloat[][Real (black) and imaginary (blue)  parts of $\mm(1+\i u)$]
{\label{fig2a}\includegraphics[height =5.5cm]{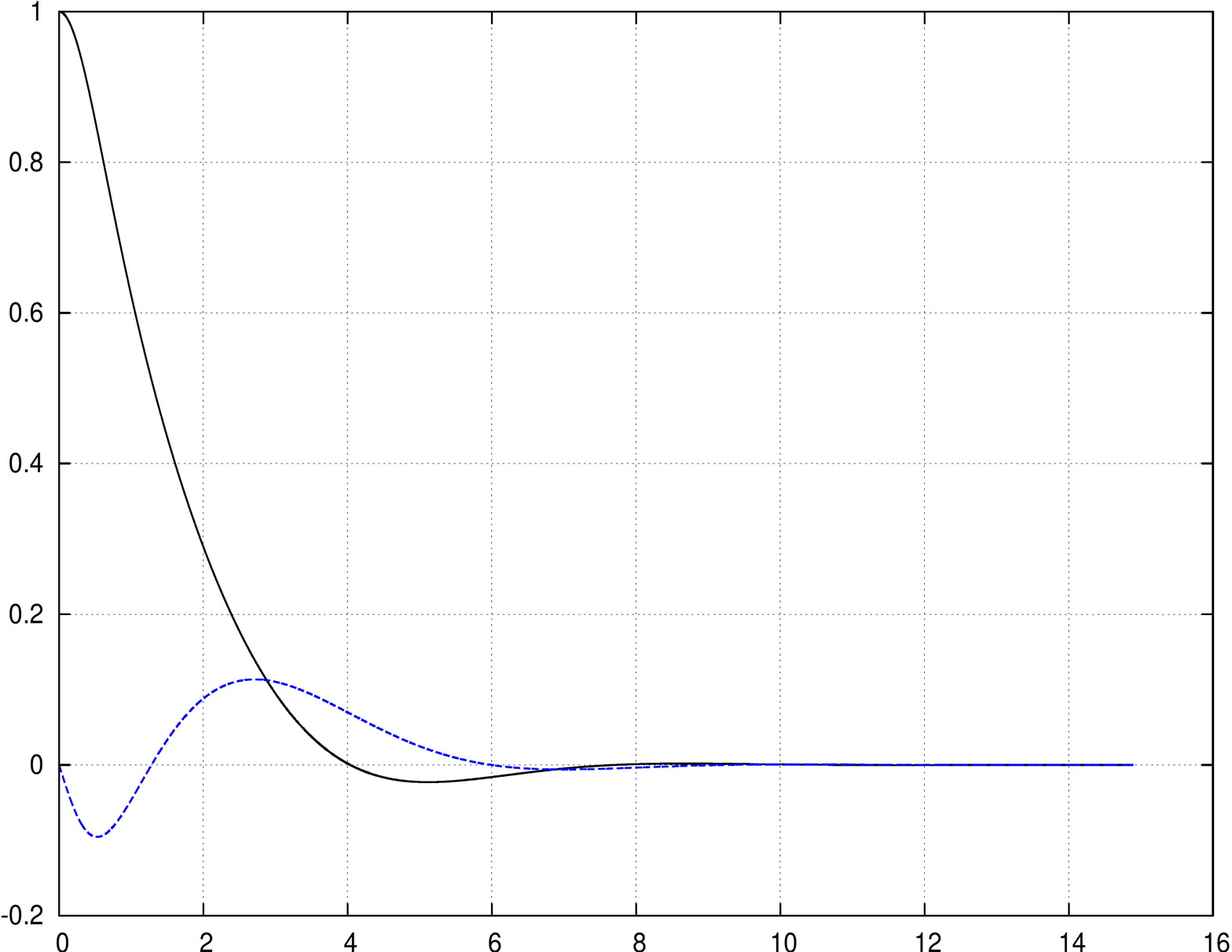}}  
\hspace{0.5cm}
\subfloat[][The density of $S_1$]
{\label{fig2b}\includegraphics[height =5.5cm]{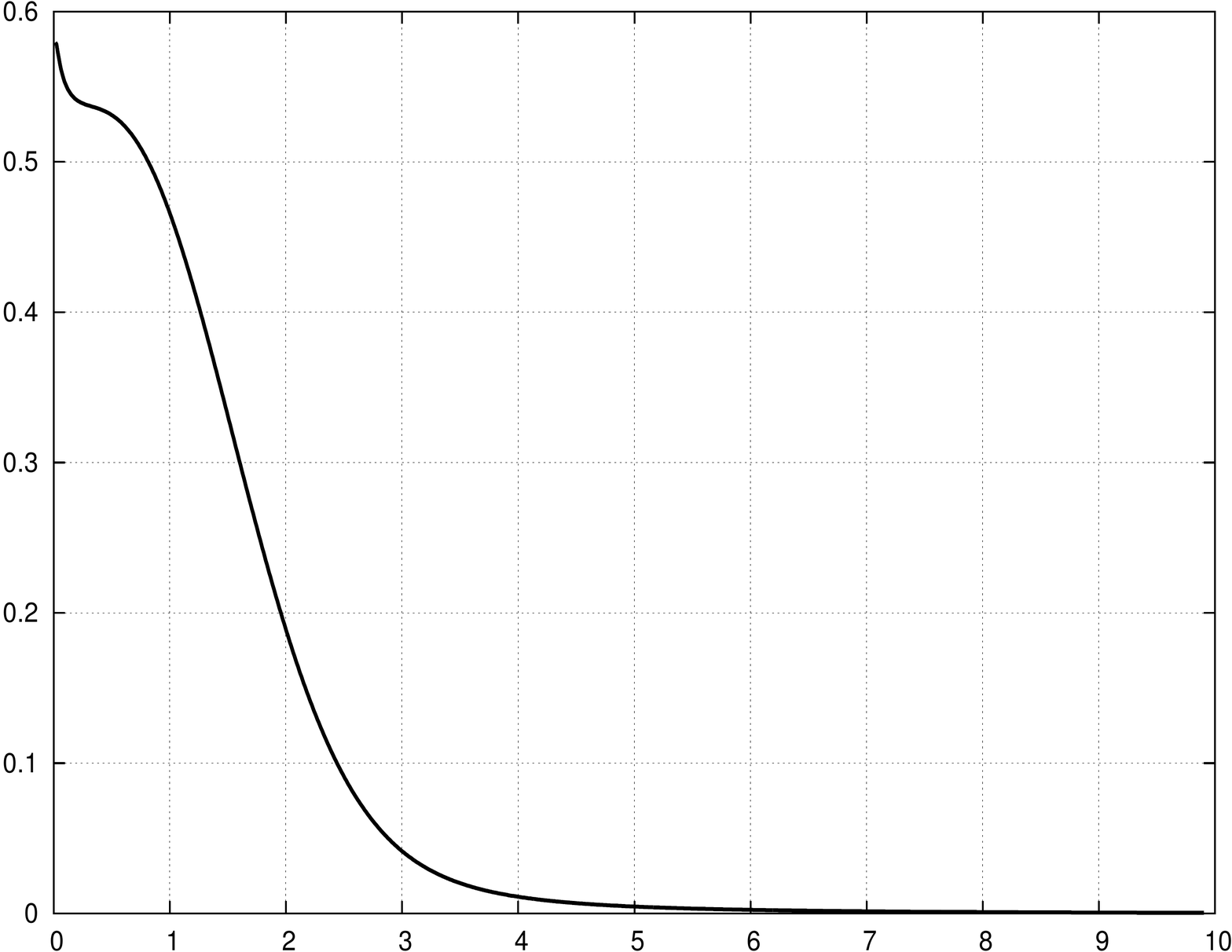}}
\caption{The Mellin transform and the density of the supremum in the case $\alpha=3/2$ and $\rho=3/5$.} 
\label{fig2}
\end{figure}

Once we are able to evaluate the Mellin transform \eqref{eqn_ms_main} numerically, it is a simple matter to find $p(x)$ by computing the inverse Mellin transform.  
Using the fact $\overline{\mm(s)}=\mm(\overline{s})$ we rewrite the inverse Mellin transform as
\beq\label{p_x_inverse_Mellin}
p(x)=\frac{1}{2\pi \i} \int\limits_{1+\i \r} \mm(s) x^{-s} \d s=\frac{1}{\pi x} \re 
\left[ \int_0^{\infty}  \mm(1+\i u) e^{-\i u \ln(x)} \d u \right].  
\eeq 
Thus we need to compute the Fourier transform of $\mm(1+\i u)$. Note that the function $\mm(1+\i u)$ is smooth, 
and as we know from Lemma 3 in \cite{Kuz2010}, it decays exponentially fast as $u\to +\infty$, thus we can truncate the integral in \eqref{p_x_inverse_Mellin} at some large number (we truncate at $u=40$) and then use Filon's method \cite{Filon} to compute this Fourier integral numerically.

The results of the computations for $\alpha=3/2$ and $\rho=3/5$ are presented on Figure \ref{fig2}. Figure \ref{fig2a} shows the real and imaginary parts of the
 Mellin transform $\mm(1+\i u)$, while on Figure \ref{fig2b} we plot the density $p(x)$. One interesting fact that 
we can observe from this picture is that the density of $S_1$ is not convex. 
This fact should be compared with the general result of Rogers \cite{Rogers01071983}, which states that if a L\'evy process 
$X$ has a completely monotone L\'evy density (such as $\Pi(\d x)$ in \eqref{def_Levy_measure}) then for any $q>0$
the density of $S_{\ee(q)}$ is also completely monotone, in particular it is convex.

%****************************************************************************************************************
%****************************************************************************************************************
%****************************************************************************************************************

\subsection{Computing $p(x)$ using the series expansions}\label{subsection_numerics_series}

%****************************************************************************************************************
%****************************************************************************************************************
%****************************************************************************************************************

In this section we investigate the second possible approach towards computing $p(x)$ - via the series expansions 
\eqref{eqn_p1}, \eqref{eqn_p2}. The coefficients of these series are not defined for $\alpha=3/2$, thus 
we perturb $\alpha=3/2$ by a small irrational number. We will first set the parameter values as
$\alpha=3/2+\sqrt{2}/50$ and $\rho=3/5$. 

The implementation of \eqref{eqn_p1}, \eqref{eqn_p2} is rather straightforward. We truncate the convergent series
\eqref{eqn_p2} at $m=n=200$ and the asymptotic series \eqref{eqn_p1} at $m=n=15$. All computations are performed in Fortran90, using quad 128-bit format, which gives us working precision of approximately 34 decimal digits.

\begin{figure}
\centering
\subfloat[]{\label{fig3a}\includegraphics[height =5.25cm]{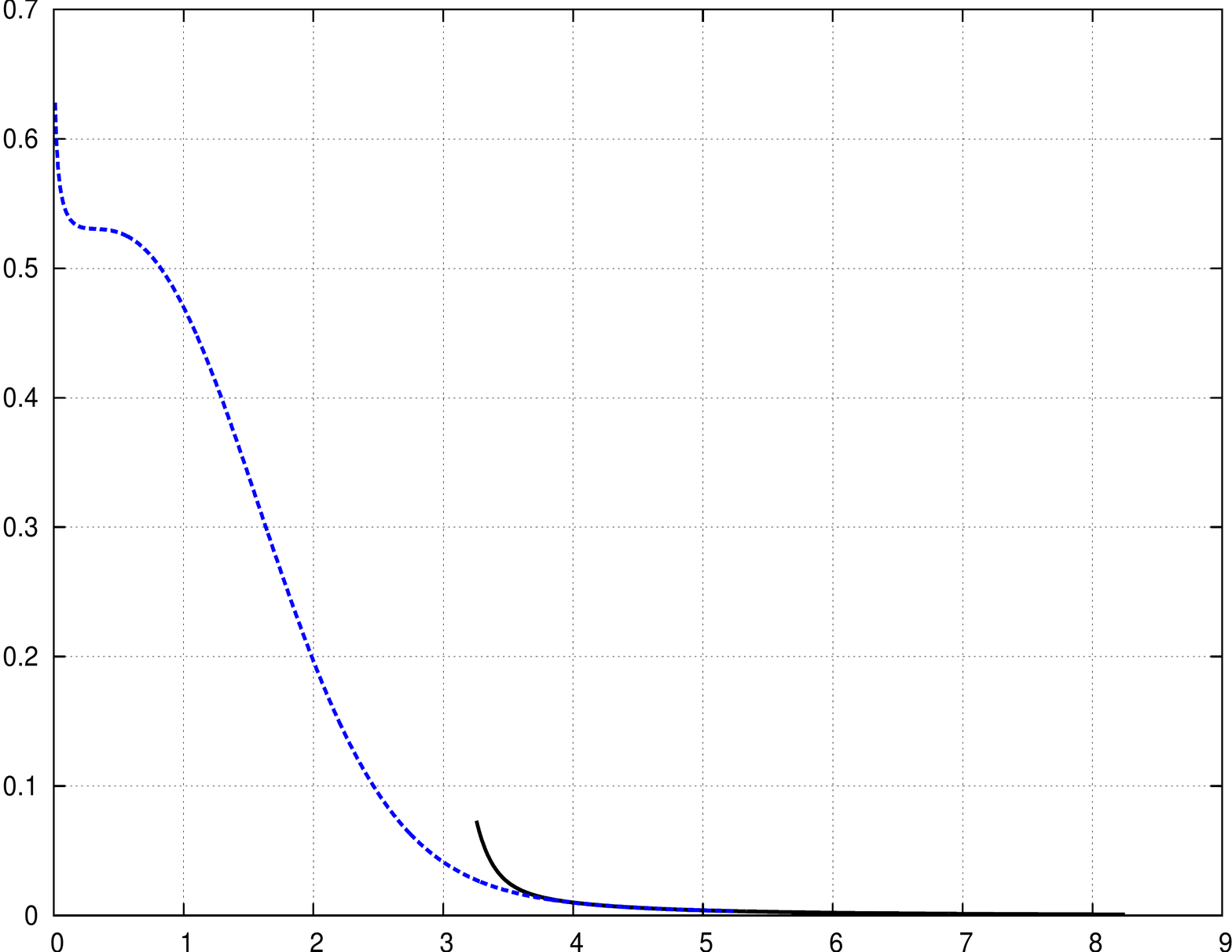}}
\hspace{0.5cm}
\subfloat[]{\label{fig3b}\includegraphics[height =5.25cm]{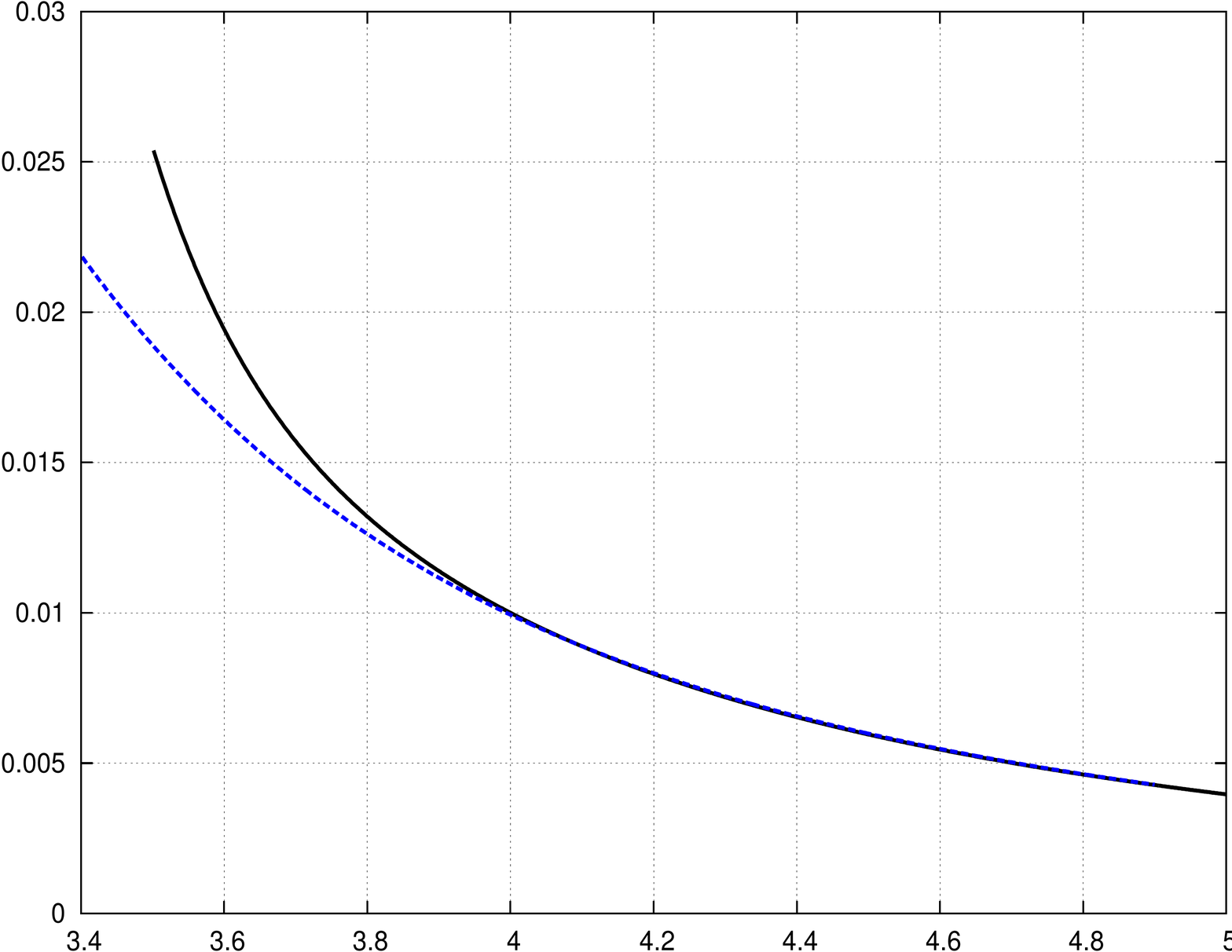}}  
\caption{Computing $p(x;\alpha,\rho)$ via convergent series \eqref{eqn_p2} (blue) and asymptotic series 
 \eqref{eqn_p1} (black). The parameters are $\alpha=3/2+\sqrt{2}/50$ and $\rho=3/5$.} 
\label{fig3}
\end{figure}

We see that the results presented on Figure \ref{fig3} are rather encouraging. First of all, the graph of $p(x)$ on Figure 
\ref{fig3a} is very similar to the one on Figure \ref{fig2b}. Second, as we see on Figure \ref{fig3b}, in the interval $x\in (4,5)$ the convergent series agrees very well with the asymptotic series.

In order to give more credibility to this method we have performed one additional experiment. 
Using the same approach as above, we have computed $p(x;\alpha,\rho)$ for $\alpha=3/2\pm \sqrt{2}/50$ and $\rho=3/5$, 
and then we have compared the average
\beqq
\tilde p(x):=\frac12 \left(p(x;3/2+\sqrt{2}/50,3/5)+p(x;3/2-\sqrt{2}/50,3/5)\right)
\eeqq
with $p(x;3/2,3/5)$ which was evaluated using the inverse Mellin transform technique described in the previous section. 
We would expect that $\tilde p(x)$ should be close to $p(x;3/2,3/5)$, since the function $p(x;\alpha,\rho)$ is continuous 
in parameter $\alpha$. This continuity property can be established using the inverse Mellin transform representation \eqref{p_x_inverse_Mellin} and the fact that 
$\mm(s)$ is continuous in $\alpha$ and decays exponentially in $|\im(s)|$ (with more work one can probably prove that 
$p(x;\alpha,\rho)$ is in fact a smooth function of all three variables $(x,\alpha,\rho)$). 
The results of this experiment are presented in Figure \ref{fig4}. We see that the error $\tilde p(x)- p(x;3/2,3/5)$ is very small, of the 
order $10^{-4}$. This confirms that both methods are in fact producing reasonably accurate results.

\begin{figure}
\centering
\subfloat[][$p(x;\alpha,\rho)$ for $\alpha=3/2+\sqrt{2}/50$ (blue),  \\ $\alpha=3/2-\sqrt{2}/50$ (red) and $\alpha=3/2$ (black) ]
{\label{fig4a}\includegraphics[height =5.25cm]{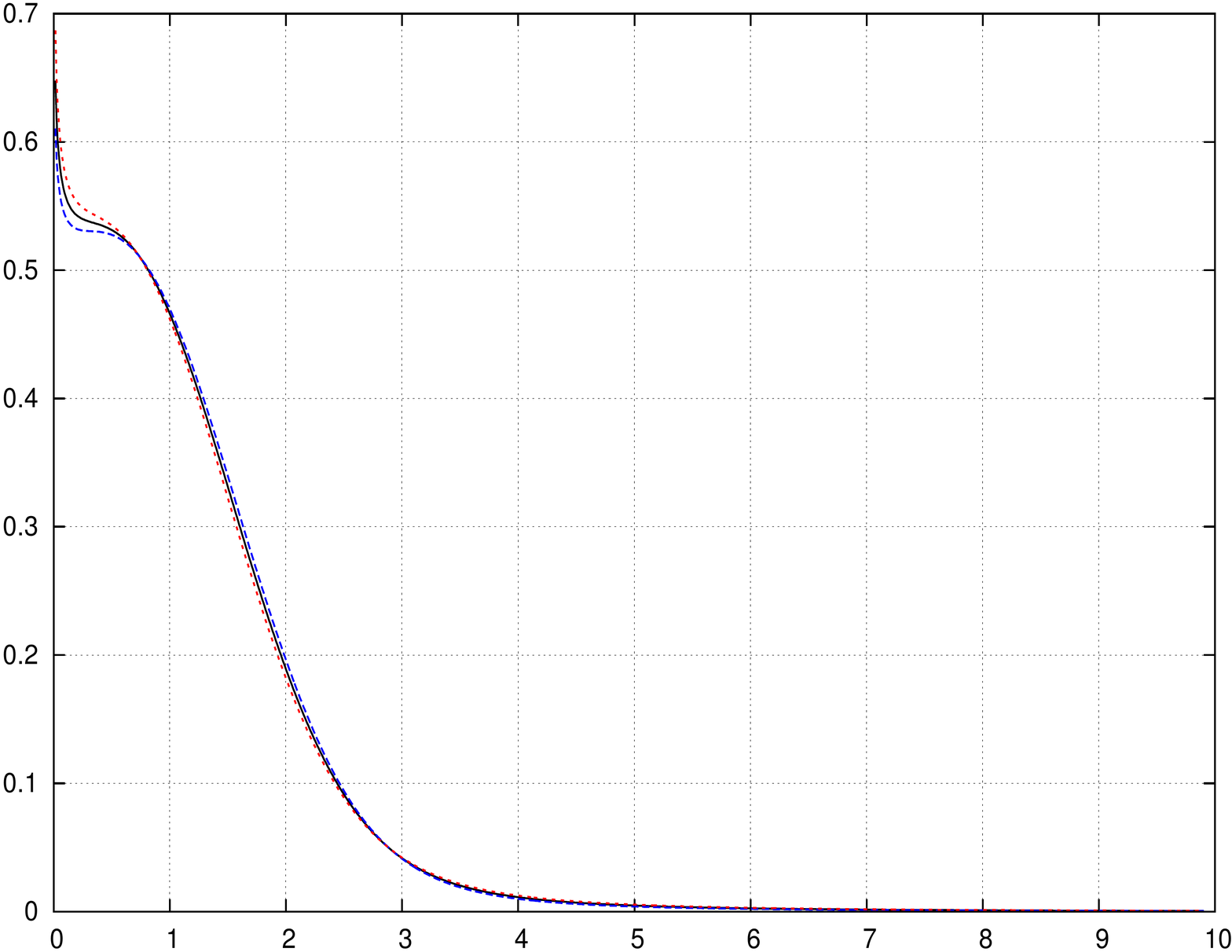}} 
\hspace{0.5cm}
\subfloat[][$\tilde p(x)-p(x;3/2,3/5)$]
{\label{fig4b}\includegraphics[height =5.25cm]{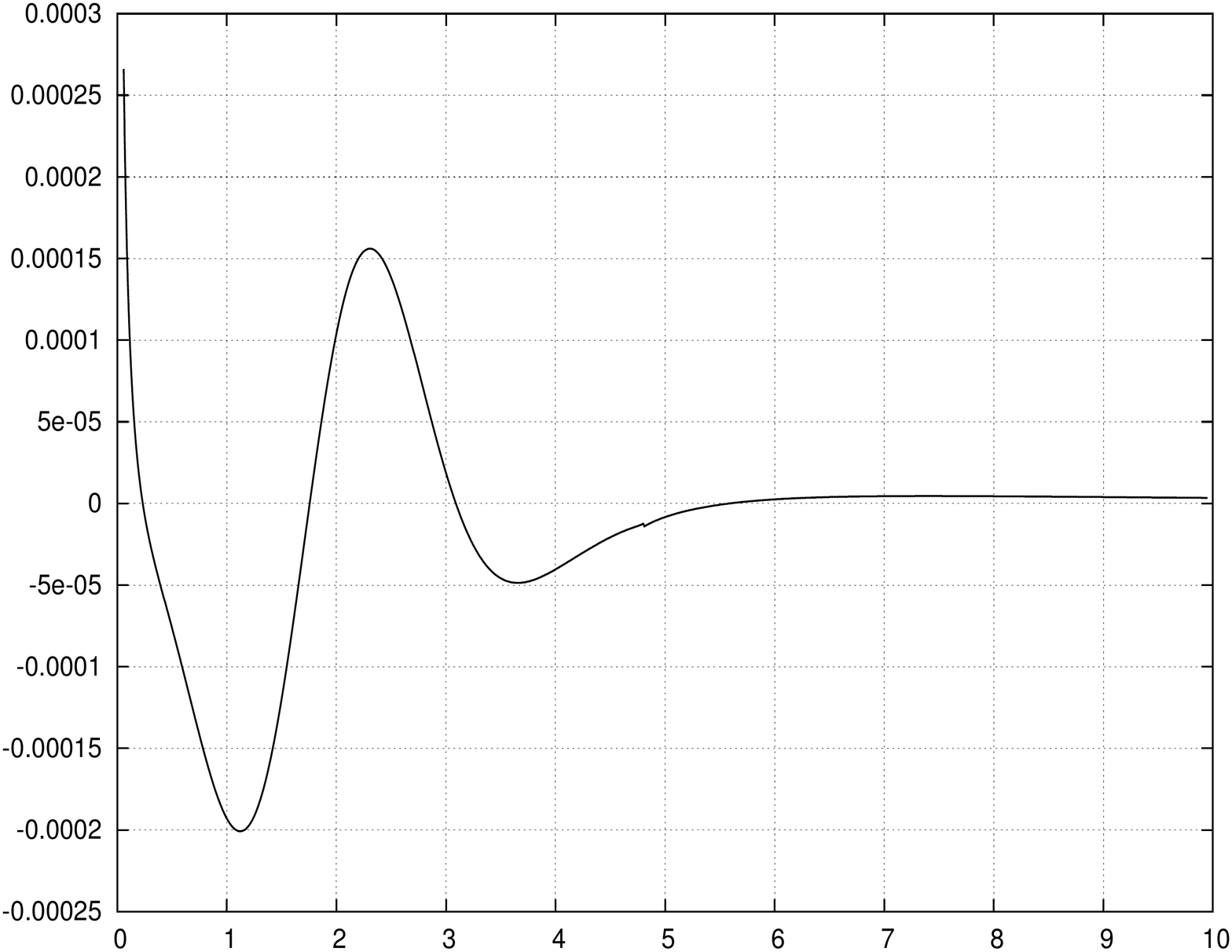}}  
\caption{Computing $p(x;\alpha,\rho)$. Here 
$\alpha \in \{3/2, 3/2\pm \sqrt{2}/50 \}$ and $\rho=3/5$. } 
\label{fig4}
\end{figure}

We would like to note that the method based on the series expansions \eqref{eqn_p1}, \eqref{eqn_p2} seems to have some serious limitations. 
In particular, we were not able to 
obtain good results in our last experiment when we modified $\alpha=3/2\pm \sqrt{2}/100$. 
It seems that when $\alpha$ is too close to a rational number, we would have to include more terms in the series \eqref{eqn_p1}, \eqref{eqn_p2} and at the same time we would have to increase the working precision.

We would also like to mention that there is another possible approach to computing $p(x)$. 
We assume that $\alpha$ is irrational. In this case the set ${\mathcal C}(\alpha)$ defined by 
\eqref{def_C_alpha} is dense, thus we can find a good approximation to $\rho$ of the form $\{l/\alpha\}$. 
This is a problem from the theory of inhomogeneous Diophantine approximations, and it can be solved efficiently using  Cassels algorithm \cite{Cassels1954}. This procedure would give us parameters $(\alpha,\tilde \rho)$, for which $\tilde \rho$ is close to $\rho$ and 
$\tilde \rho \in {\mathcal C}(\alpha)$. The process $\tilde X$ defined by parameters $(\alpha,\tilde \rho)$ belongs to Doney 
class ${\mathcal C}_{k,l}$ for some integers $k$ and $l$, 
therefore we can compute $p(x;\alpha,\tilde \rho)$ using the simpler series expansions given in Theorem 10 in \cite{Kuz2010}. 
The important fact is that the coefficients of these series expansions involve simple bounded products of trigonometric functions. 
While we did not implement this method, we expect that 
in this case the computation can be performed more efficiently and may not even require high precision arithmetic.

%****************************************************************************************************************
%****************************************************************************************************************
%****************************************************************************************************************

\section{Conclusion and several open problems}\label{section_conclusion}

%****************************************************************************************************************
%****************************************************************************************************************
%****************************************************************************************************************

In conclusion we would like to discuss several connections that the study of the distribution of extrema of stable processes has to other areas of Mathematics and Mathematical Physics. 

First of all, the series expansions of the density of $S_1$ are related to $q$-series (see an excellent book by Gasper and Rahman \cite{Gasper_Rahman}). The $q$-series can be informally defined as hypergeometric series where the terms involving Gamma function 
and factorials are replaced by the q-Pochhammer symbol, defined as
\beqq
(a;q)_n := \prod\limits_{k=0}^{n-1} (1-a q^k).
\eeqq
It is clear that the finite products involving $\sin(\cdot)$ function in the definition \eqref{def_a_mn} of coefficients $a_{m,n}$
can be rewritten in terms of the $q$-Pochhammer symbol, thus the series 
\eqref{eqn_p1}, \eqref{eqn_p2} can be considered as $q$-series. 

Convergence properties of $q$-series when $|q|=1$ is an interesting question that has attracted the interest of many researchers. 
The first results in this area were obtained by Hardy and Littlewood \cite{Hardy_Littlewood}, where they have  investigated 
the series 
\beqq
\sum\limits_{n\ge 0} \frac{z^n}{(q;q)_n}.
\eeqq
Driver et. al.  \cite{Driver1991469} have investigated the convergence properties of the series
\beqq
\sum\limits_{n\ge 0} (a;q)_n z^n,
\eeqq
and Lubinsky \cite{Lubinsky} and later Buslaev \cite{Buslaev} have studied the Rogers-Ramanujan function
\beqq
G_q(z)=1+\sum\limits_{n\ge 1} \frac{q^{n^2} z^n}{(q;q)_n}.
\eeqq
This function was central in Lubinsky's disproof of the Baker-Gammel-Wills conjecture (see \cite{Lubinsky}). 

The convergence properties of $q$-series on the unit circle $|q|=1$ are naturally related to Diophantine approximations and properties of 
continued fraction expansion of $\arg(q)$. As we have seen in the proof of Theorem \ref{main_theorem}, in our case the situation is even 
more delicate and the convergence depends essentially on the properties of both parameters $\alpha$ and $\rho$. 
The essential tool in the proof of Theorem \ref{main_theorem} was the result bu Bugeaud \cite{Bugeaud2003} from the theory of inhomogeneous Diophantine approximations. 

There also exists an intriguing connection between the Wiener-Hopf factorization of stable processes and 
 certain special function called {\it quantum dilogarithm}, which has applications in Quantum Topology and Cluster Algebra Theory. For $b\in \c$ with $\re(b)>0$ the Faddeev's quantum dilogarithm is defined 
in the strip $|\im(z)|<\frac{1}{2} \re(b+b^{-1})$ by
\beqq
\Phi_b(z):=\exp\left( -\frac{1}{4} \int\limits_{\r} 
\frac{e^{-2 \i z x}}{\sinh(xb) \sinh(x/b)} \frac{\d x}{x} \right), 
\eeqq
where the singularity at $x=0$ is circled from above, 
see \cite{Faddeev}, \cite{Kashaev} and the references therein. The function $z \mapsto \Phi_b(z)$ can be analytically continued to a meromorphic function on the entire complex plane.  Combining formulas (4.6) in \cite{Kashaev}  and  (4.2) in \cite{Kuz2010} we obtain a particularly simple 
expression for the positive Wiener-Hopf factor in terms of Faddeev's quantum dilogarithm
\beqq
\phi(\exp(2\pi w);\alpha,\rho)=\frac{\Phi_{\sqrt{\alpha}}\left( \sqrt{\alpha} \left( w - \frac{\i \rho}2 \right) \right) }
{\Phi_{\sqrt{\alpha}}\left( \sqrt{\alpha} \left( w + \frac{\i \rho}2 \right) \right) }.
\eeqq 

\vspace{0.2cm}
 We would like to conclude this paper by stating three open problems. 

\vspace{0.2cm}
\noindent
{\bf Problem 1.}  Doney's ${\mathcal C}_{k,l}$ classes are well-understood from the analytical point of view. 
For example, if we consider the Wiener-Hopf factor $\phi(\exp(w))$ or the Mellin transform $\mm(s)$, it is known that the zeros/poles of 
these functions lie on certain lattices, and when $(\alpha,\rho)$ satisfy \eqref{Ckl} these lattices overlap and almost all zeros/poles are cancelled 
(see Lemma 2 and the discussion after Corollary 2 in \cite{Kuz2010}). This explains why the general formula
for the Wiener-Hopf factor (see (4.11) in \cite{Kuz2010}) or the general formula \eqref{formula_Ms_general} for the Mellin transform  can be simplified when $X\in{\mathcal C}_{k,l}$. 
The spectrally-positive \{ spectrally-negative \} processes correspond to Doney's class ${\mathcal C}_{0,1}$ \{ resp.  ${\mathcal C}_{-1,-1}$ \}, 
thus it seems that the classes ${\mathcal C}_{k,l}$ should be considered as generalizations of spectrally one-sided processes. 
An important problem is to a give probabilistic interpretation of Doney's classes. 
Hopefully, an answer to this question would also provide a probabilistic and more insightful derivation  of the existing results on the Wiener-Hopf factorization, Mellin transform and the density of supremum in the case when $X\in {\mathcal C}_{k,l}$.

\vspace{0.2cm}
\noindent
{\bf Problem 2.} In Theorem \ref{main_theorem} we have established that the double series \eqref{eqn_p1}, \eqref{eqn_p2} do
 not converge absolutely by showing that there is an infinite subsequence of the terms of the series which 
 is bounded away from zero. However, this double series could still converge conditionally to $p(x)$. 
 Is it possible to find a way to order the partial sums of the series  \eqref{eqn_p1}, \eqref{eqn_p2} so that they converge to $p(x)$ 
 {\it for all } irrational $\alpha$?

\vspace{0.2cm}
\noindent
{\bf Problem 3.} From Theorems \ref{thm_main} and \ref{main_theorem}  we know that when 
 $\alpha \notin {\mathcal L} \cup \q$ the series  \eqref{eqn_p1}, \eqref{eqn_p2} converge absolutely for all $\rho$, 
 and when $\alpha \in \tl$ these series do not converge absolutely for almost all $\rho$. This situation is reminiscent of Kolmogorov's zero-one law. Can this fact be established rigorously? Let us state this question more precisely: does there exist a value of $\alpha$ for which the series \eqref{eqn_p1}, \eqref{eqn_p2} converge absolutely on a set of $\rho$ of positive Lebesgue measure, and at the same time do not converge absolutely on a set of $\rho$ of positive Lebesgue measure?

%**************************************************************************************************
%**************************************************************************************************
%**************************************************************************************************

\bibliographystyle{abbrv}
\bibliography{literature}

%**************************************************************************************************
%**************************************************************************************************
%**************************************************************************************************

%****************************************************************************************************************
%****************************************************************************************************************
%****************************************************************************************************************

\end{document}